\documentclass{amsart}
\usepackage{amsmath} 
\usepackage{amssymb} 
\usepackage{amsthm} 
	\theoremstyle{definition}
	\newtheorem{thm}{Theorem}[section]
	\newtheorem{prop}[thm]{Proposition}
	\newtheorem{coro}[thm]{Corollary}
	\newtheorem{lemma}[thm]{Lemma}
	\newtheorem{defn}[thm]{Definition}

	\newtheorem{nota}[thm]{Notation}

\usepackage{chngcntr} 
	\counterwithin{figure}{section} 
\usepackage{cite} 
\usepackage{enumerate} 
\usepackage{etex} 
\usepackage{fancyhdr} 
\usepackage{float} 
\usepackage[utf8]{inputenc} 
\usepackage[margin=1.2in]{geometry} 
\usepackage{graphicx} 
\usepackage{multicol} 
\usepackage{pgfplots} 
\usepackage{subcaption} 
\usepackage{tikz} 
\usepackage{tikz-cd} 
\usepackage{xcolor} 

\newcommand{\bT}{\mathbb{T}}

\newcommand{\bZ}{\mathbb{Z}}

\newcommand{\sK}{\mathcal{K}}

\newcommand{\sP}{\mathcal{P}}

\newcommand{\sS}{\mathcal{S}}
\newcommand{\sT}{\mathcal{T}}

\newcommand{\eps}{\varepsilon}
\newcommand{\vphi}{\varphi}

\newcommand{\setdiv}{\,\big|\,}
\newcommand{\ran}{\text{ran}}
\newcommand{\id}{\text{id}}
\newcommand{\orb}{\text{orb}}

\newcommand{\dprime}{^{\prime\prime}}

\newcommand{\systemfull}{\sS=( T,( X_t )_{ t = 1 , \ldots , T }, ( K_t )_{ t = 1 , \ldots , T }, ( Y_{ t , k } )_{ t = 1 , \ldots , T ; k = 1 , \ldots , K_t } , ( J_{ t , k } )_{ t = 1 , \ldots , T ; k = 1 , \ldots , K_t } ) }
\newcommand{\systembasic}{\sS=( T, ( X_t ) , ( K_t ), ( Y_{ t , k } ) , ( J_{ t , k } ) ) }
\newcommand{\systemarg}[1]{ \sS^{ #1 } = ( T^{ #1 } , ( X_t^{ #1 } ) , ( K_t^{ #1 } ) , ( Y_{ t , k }^{ #1 } ) , ( J_{ t , k }^{ #1 } ) ) }

\newcommand{\image}{\text{im}}

\address{Department of Mathematics, Ben Gurion University of the Negev,
P.O.B. 653, Be'er Sheva 84105, Israel}
\email{PHerstedt@gmail.com}

\allowdisplaybreaks

\begin{document}

\title[Dynamical classification for fiberwise essentially minimal systems]{A dynamical classification for crossed products of fiberwise essentially minimal zero-dimensional dynamical systems}

\author{Paul Herstedt}
\thanks{This research was supported by the Israel Science Foundation grant no.~476/16.}

\begin{abstract}
We prove that crossed products of fiberwise essentially minimal zero-dimensional dynamical systems have isomorphic $ K $-theory if and only if the dynamical systems are strong orbit equivalent. Under the additional assumption that the dynamical systems have no periodic points, this gives a classification theorem including isomorphism of the $ C^* $-algebras as well. We additionally explore the $ K $-theory of such crossed products and the Bratteli diagrams associated to the dynamical systems.
\end{abstract}

\maketitle

\section*{Introduction}

In 1990, Ian Putnam proved in \cite{Putnam90} that the crossed product $ C^* $-algebras associated to minimal Cantor systems are A$\bT$-algebras of real rank zero. Using the classification results of Geroge Elliott in \cite{Elliott93} and Dadarlat-Gong in \cite{DadarlatGong97}, one sees that such $ C^* $-algebras are classifiable by their $ K $-theory. In 1995, Putnam, along with Thierry Giordano and Christian Skau, expanded this classification theorem to include dynamics; in \cite{GiordanoPutnamSkau95}, they showed that there is a condition on the dynamical systems, called ``strong orbit equivalence", that is equivalent to isomorphism of the crossed product $ C^* $-algebras. This dynamical classification was motivated by Krieger's theorem (see \cite{Krieger69} and \cite{Krieger76}), which says that for ergodic non-singular systems, the associated von Neumann crossed product factors are isomorphic if and only if the systems are orbit equivalent. The goal of this paper is to expand these results in the $ C^* $-setting.

In our previous paper \cite{Herstedt21}, we determined a condition on a zero-dimensional dynamical system called ``fiberwise essentially minimal" (see Definition \ref{defnFiberwiseEssentiallyMinimal} below) that guarantees that the associated crossed product is an A$\bT$-algebra. As its name suggests, this class is a broadening of minimal (and also essentially minimal). If we further restrict the system to have no periodic points, this crossed product has real rank zero and is therefore classifiable by $ K $-theory (due to the work of Elliott in \cite{Elliott93} and Dadarlat-Gong in \cite{DadarlatGong97}). This was an expansion of work done on the minimal Cantor case in 1990 by Ian Putnam (see \cite{Putnam89} and \cite{Putnam90}) in which the crossed products are simple, and work done on the essentially minimal case in 1992 by Putnam and Skau along with Richard Herman (see \cite{HermanPutnamSkau92}) in which the crossed products are not necessarily simple. Our result from \cite{Herstedt21} includes many more non-simple crossed products.

This paper expands on the work in \cite{Herstedt21} in two major ways. The first of which is what we explore in Section 2, where we discuss some specifics about the $ K $-theory of the crossed products. We define ``large subalgebras" of our crossed products (see Definition \ref{defnAZ}) which are AF-subalgebras (see Theorem \ref{thmAZisAF}) that have the same $ K_0 $ group as the crossed product (see Theorem \ref{thmK0AZIsomorphism}). This mirrors the result of large subalgebras in the minimal case by Putnam in \cite{Putnam90}. We also give a simple description of the $ K_1 $ group of the crossed product in Theorem \ref{thmK1Formula}.

The second major aspect of this paper is expanding the dynamical classification of minimal Cantor systems that coincides with the $ K $-theoretic classification, introduced by Giordano, Putnam, and Skau in \cite{GiordanoPutnamSkau95}. They introduce the notion of ``strong orbit equivalence", which we expand to the fiberwise essentially minimal case in Definition \ref{defnFiberwiseEssentiallyMinimal}. In Section 3, we discuss how the circle algebra direct system that gives the A$\bT$-algebra of the crossed product gives us an ordered Bratteli diagram whose Vershik system is conjugate to the original dynamical system. Using this Bratteli diagram along with our $ K $-theory results, we then prove Theorem \ref{thmMainTheoremPrelude}, which tells us that for fiberwise essentially minimal zero-dimensional systems, $ K $-theory isomorphism of the crossed products is equivalent to strong orbit equivalence of the dynamical systems. This, combined with the classification result of \cite{Elliott93} and \cite{DadarlatGong97}, gives us Theorem \ref{thmMainTheorem}, which tells us that if the dynamical systems have no periodic points, this is also equivalent to isomorphism of the crossed products.

\textbf{Acknowledgements.} This work was done at the Ben Gurion University of the Negev, funded by the Israel Science Foundation (grant no.476/16).

\section{Preliminaries}

This section introduces terms that will be used to prove the main theorems of the paper, along with some examples and some previous relevant results.

Let $ X $ be a totally disconnected compact metrizable space and let $ h : X \to X $ be a homeomorphism of $ X $. We call $ ( X , h ) $ a \emph{zero-dimensional system}. Let $ \alpha $ be the automorphism of $ C ( X ) $ induced by $ h $; that is, $ \alpha $ is defined by $ \alpha ( f ) ( x ) = f ( h^{ -1 } ( x ) ) $ for all $ f \in C ( X ) $ and all $ x \in X $. Then we denote the crossed product of $ C ( X ) $ by $ \alpha $ by $ C^* ( \bZ , X , h ) $ (or, less commonly, $ C^* ( \bZ , C ( X ) , \alpha ) $). We denote the ``standard unitary" of $ C^* ( \bZ , X , h ) $ by $ u $, which is a unitary element of $ C^* ( \bZ , X , h ) $ that satisfies $ u f u^* = \alpha ( f ) $ for all $ f \in C ( X ) $.

We will use the disjoint union symbol $ \bigsqcup $ to denote unions of disjoint sets. We will not always say explicitly that the sets in this union are disjoint, as this will be implied by the notation. By a partition $ \sP $ of $ X $, we mean a finite set of mutually disjoint compact open subsets of $ X $ whose union is $ X $.

We say that a nonempty closed subset $ Y $ of $ X $ is a \emph{minimal set} if it is $ h $-invariant and has no nonempty $ h $-invariant proper closed subsets. By Zorn's lemma, minimal sets exist for every zero-dimensional system. We say that a dynamical system $ ( X , h ) $ is an \emph{essentially minimal system} if it has a unique minimal set. 

The following definition is introduced as Definition 1.10 in \cite{Herstedt21}.

\begin{defn}\label{defnSystemOfFiniteReturnTimeMaps}
Let $ ( X , h ) $ be a zero-dimensional system and let $ \sP $ be a partition of $ X $. We define a 
\emph{system of finite first return time maps subordinate to $ \sP $} to be a tuple
\[
\systemfull
\]
such that:
\begin{enumerate}[(a)]
\item We have $ T \in \bZ_{ > 0 } $. 
\item For each $ t \in \{ 1 , \ldots , T \} $, $ X_t $ is a compact open subset of $ X $. That $ \sS $ is subordinate to $ \sP $ means that, for each $ t \in \{ 1 , \ldots , T \}$, $ X_t $ is contained in an element of $ \sP $.
\item For each $ t \in \{ 1 , \ldots , T \} $, $ K_t \in \bZ_{ > 0 } $. 
\item For each $ t \in \{ 1 , \ldots , T \} $ and each $ k \in \{ 1 , \ldots , K_t \} $, $ Y_{ t , k } $ is a compact open subset of $X_t$. Moreover, for each $ t \in \{ 1 , \ldots , T\}$, $ \{ Y_{ t , 1 } , \ldots , Y_{ t , K_t } \} $ is a partition of $X_t$; that is, 
\[
\bigsqcup_{ k = 1 }^{ K_t } Y_{ t , k } = X_t .
\]
\item For each $ t \in \{ 1 , \ldots , T \} $ and each $ k \in \{ 1 , \ldots , K_t \} $, $ J_{ t , k } \in \bZ_{ > 0 } $. Using Definition \ref{defnLambdaU}, we can write $ \{ J_{ t , k } \} = \lambda_{ X_t } ( Y_{ t , k } ) $. Moreover, for each $ t \in \{ 1 , \ldots , T \} $, $ \{ h^{ J_{ t , 1 } } ( Y_{ t , 1 } ) , \ldots , h^{ J_{ t , K_t } } ( Y_{ t , K_t } ) \} $ is a partition of $ X_t $; that is, 
\[
\bigsqcup_{ k = 1 }^{ K_t } h^{ J_{ t , k } } ( Y_{ t , k } ) = X_t .
\]
\item The set
\[
\sP_1 ( \sS ) = \big\{ h^j ( Y_{ t , k } ) \setdiv \mbox{$ t \in \{ 1 , \ldots , T \} $, $ k \in \{ 1 , \ldots , K_t \} $, and $ j \in \{ 0 , \ldots , J_{ t , k } - 1 \} $} \big\}
\]
is a partition of $X$. Note that this combined with condition (e) also implies 
\[
\sP_2 ( \sS ) = \big\{ h^j ( Y_{ t , k } ) \setdiv \mbox{$ t \in \{ 1 , \ldots , T \} $, $ k \in \{ 1 , \ldots , K_t \} $, and $ j \in \{ 1 , \ldots , J_{ t , k } \} $} \big\}
\]
is a partition of $X$.
\end{enumerate}
\end{defn}

The following definition is introduced as Definition 1.18 in \cite{Herstedt21}.

\begin{defn}\label{defnFiberwiseEssentiallyMinimal}
Let $ ( X , h ) $ be a zero-dimensional system and let $ Z \subset X $ be a closed subset. We say that the triple $ ( X , h , Z ) $ is a \emph{fiberwise essentially minimal zero-dimensional system} if there is a quotient map $ \psi: X \to Z $ such that
\begin{enumerate}[(a)]
\item $ \psi|_Z : Z \to Z $ is the identity map.
\item $ \psi \circ h = \psi $.
\item For each $ z \in Z $, $ ( \psi^{ -1 } ( z ) , h|_{ \psi^{ -1 } ( z ) } ) $ is an essentially minimal system and $ z $ is in its minimal set.
\end{enumerate}
\end{defn}

For examples of fiberwise essentially minimal zero-dimensional systems, see Examples 1.18 in \cite{Herstedt21}. The connection between Definition \ref{defnSystemOfFiniteReturnTimeMaps} and Definition \ref{defnFiberwiseEssentiallyMinimal} is the following theorem, which appears as Theorem 2.1 of \cite{Herstedt21}.

\begin{thm}\label{thmFiberwiseIffAdmitsPartitions}
Let $ ( X , h ) $ be a zero-dimensional system. Then there exists some closed $ Z \subset X $ such that $ ( X , h , Z ) $ is fiberwise essentially minimal if and only if for any partition $ \sP $ of $ X $, $ ( X , h ) $ admits a system of finite first return time maps subordinate to $ \sP $.
\end{thm}

The following is Theorem 2.2 in \cite{Herstedt21}.

\begin{thm}\label{thmPreviousMainTheorem}
Let $ ( X , h , Z ) $ be a fiberwise essentially minimal zero-dimensional system. Then $ C^* ( \bZ , X , h ) $ is an A$\bT$-algebra.
\end{thm}

The following is a consequence of the proof of Theorem \ref{thmPreviousMainTheorem}. By ``circle algebra", we mean an algebra isomorphic to a finite direct sum of matrices and matrices over $ C ( S^1 ) $.

\begin{coro}\label{coroCircleAlgebraApproximation}
Let $ ( X , h , Z ) $ be a fiberwise essentially minimal zero-dimensional system, let $ \sP $ be a partition of $ X $, let $ a_1 , \ldots , a_n \in C^* ( \bZ , X , h ) $, and let $ \eps > 0 $. Then there is a circle algebra $ A \subset C^* ( \bZ , X , h ) $ and a partition $ \sP' $ of $ X $ that is finer than $ \sP $ such that 
\begin{enumerate}[(a)]
\item The diagonal matrices of $ A $ are $ C ( \sP' ) $.
\item For each $ k \in \{ 1 , \ldots , n \} $, there is a $ b_k \in A $ such that $ \| a_k - b_k \| < \eps $.
\end{enumerate}
\end{coro}

We now introduce the concepts important to the dynamical side of the discussion in this paper. Let $ ( X_1 , h_1 ) $ and $ ( X_2 , h_2 ) $ be dynamical systems. By an \emph{orbit map}, we mean a homeomorphism $ F : X_1 \to X_2 $ such that for all $ x \in X_1 $, we have $ F ( \orb_{ h_1 } ( x ) ) = \orb_{ h_2 } ( F ( x ) ) $, where $ \orb_{ h_1 } ( x ) $ denotes the $ h_1 $-orbit of $ x $ (and likewise for $ \orb_{ h_2 } $). We say that $ ( X_1 , h_1 ) $ and $ ( X_2 , h_2 ) $ are \emph{orbit equivalent} if there exists such an $ F $. If the orbit map satisfies $ F \circ h_1 = h_2 \circ F $, we say that $ ( X_1 , h_1 ) $ and $ ( X_2 , h_2 ) $ are \emph{conjugate}. When we consider orbit maps between $ ( X_1 , h_1 , Z_1 ) $ and $ ( X_2 , h_2 , Z_2 ) $ for closed sets $ Z_1 \subset X_1 $ and $ Z_2 \subset X_2 $, we require that $ F ( Z_1 ) = Z_2 ) $.

\begin{defn}
Let $ ( X_1 , h_1 ) $ and $ ( X_2 , h_2 ) $ be dynamical systems and let $ F : X_1 \to X_2 $ be an orbit map. Then there are functions $ \beta, \gamma : X_1 \to \bZ $, called \emph{orbit cocycles}, that satisfy $ ( F \circ h_1^{ \beta ( x ) } ) ( x ) = ( h_2 \circ F ) ( x ) $ and $ ( h_2^{ \gamma ( x ) } \circ F ) ( x ) = ( F \circ h_1 ) ( x ) $.
\end{defn}

The following is a generalization of Definition 1.3 of \cite{GiordanoPutnamSkau95} from the minimal case to the fiberwise essentially minimal case. 

\begin{defn}
Let $ ( X_1 , h_1 , Z_1 ) $ and $ ( X_2 , h_2 , Z_2 ) $ be fiberwise essentially minimal zero-dimensional systems. We say that $ ( X_1 , h_1 , Z_1 ) $ and $ ( X_2 , h_2 , Z_2 ) $ are \emph{strong orbit equivalent} if there is an orbit map $ F : X_1 \to X_2 $ such that the orbit cocycles $ \beta , \gamma : X_1 \to \bZ $ are continuous on $ X_1 \setminus Z_1 $.
\end{defn}

\section{$K$-Theory}

In this section we discuss the $ K $-theory of the crossed products associated to fiberwise essentially minimal zero-dimensional systems. For some references on work done in the minimal case, see \cite{Putnam90} and \cite{GiordanoPutnamSkau95}. For a reference on work done in the essentially minimal case, see \cite{HermanPutnamSkau92}.

\begin{defn}\label{defnOrderedGroup}
An \emph{ordered group} is a pair $ ( G , G^+ ) $, where $ G $ is a countable abelian group and $ G^+ $ is a subset of $ G $, called the \emph{positive cone}, that satisfies the following:
\begin{enumerate}[(a)]
\item For all $ g_1 , g_2 \in G^+ $, we have $ g_1 + g_2 \in G^+ $.
\item For all $ g \in G $, there are $ g_1 , g_2 \in G^+ $ such that $ g = g_1 - g_2 $.
\item The identity of $ G $ is the only element in both $ G^+ $ and $ -G^+ $.
\end{enumerate}
We call $ e \in G^+ $ an \emph{order} unit if for all $ g \in G^+ $, there is some $ n \in \bZ_{ > 0 } $ such that $ n e - g \in G^+ $. 
\end{defn}

Given an ordered group $ ( G , G^+ ) $, we may write $ g \geq 0 $ to denote $ g \in G^+ $. The notation $ g_1 \geq g_2 $ means that $ g_1 - g_2 \in G^+ $. By a homomorphism of ordered groups $ ( G_1 , G_1^+ ) $ and $ ( G_2 , G_2^+ ) $, we mean a homomorphism of groups $ \vphi : G_1 \to G_2 $ such that $ \vphi ( G_1^+ ) \subset G_2^+ $.

When we fix a particular order unit $ e \in G^+ $, we may write the triple $ ( G , G^+ , e ) $ and call this an ordered group with distinguished order unit. By a homomorphism of ordered groups with distinguished order units $ ( G_1 , G_1^+ , e_1 ) $ and $ ( G_2 , G_2^+ , e_2 ) $, we mean a homomorphism of ordered groups $ \vphi : G_1 \to G_2 $ such that $ \vphi ( e_1 ) = e_2 $.

We introduce notation important to the following proposition, which is Proposition 2.2 of \cite{Herstedt21}, and a direct consequence of Theorem 2.4 of \cite{PimsnerVoiculescu80}. Let $ \sT $ denote the Toeplitz algebra, the universal $ C^* $-algebra generated by a single isometry $ s $. Let $ A $ be a unital $ C^* $-algebra and let $ \alpha $ be an automorphism of $ A $, and let $ u $ be the standard unitary of $ C^*( \bZ , A , \alpha ) $. We denote by $ \sT ( A , \alpha ) $ the Toeplitz extension of $ A $ by $ \alpha $, which is the subalgebra of $ C^* ( \bZ , A , \alpha ) \otimes \sT $ generated by $ A \otimes 1 $ and $ u \otimes s $. The ideal generated by $ A \otimes ( 1 - s s^* ) $ is isomorphic to $ A \otimes \sK $, and the quotient by this ideal is isomorphic to $ C^* ( \bZ , A , \alpha) $.

\begin{prop}\label{propKTheoryExactSequence}
Let $ ( X , h ) $ be a zero-dimensional system. Let $ \alpha $ be the automorphism of $ C ( X ) $ induced by $ h $; that is, $ \alpha $ is defined by $ \alpha ( f ) ( x ) = f ( h^{ -1 } ( x ) ) $ for all $ f \in C ( X ) $ and all $ x \in X $. Let $ \delta $ be the connecting map obtained from the exact sequence
\begin{center}
\begin{tikzcd}
0 \arrow[r] & C(X) \otimes \sK \arrow[r] & \sT(C(X),\alpha) \arrow[r] & C^*(\bZ,A,\alpha) \arrow[r] & 0,
\end{tikzcd}
\end{center}
where $ K_0 ( C ( X ) \otimes \sK ) $ is identified with $ K_0 ( C ( X ) ) $ in the standard way. Let $ i :   C(X) \to C^* ( \bZ , X , h ) $ be the natural inclusion. Then there is an exact sequence 
\begin{center}
\begin{tikzcd}
0 \arrow[r] & K_1(C^*(\bZ,X,h)) \arrow[r, "\delta"] & K_0(C(X)) \arrow[r, "\id-\alpha_*"] & K_0(C(X)) \arrow[r, "i_*"] & K_0(C^*(\bZ,X,h)) \arrow[r] & 0.
\end{tikzcd}
\end{center}
\end{prop}

\begin{proof}
Since $ K_1 ( C ( X ) ) = 0 $, this follows immediately from Theorem 2.4 of \cite{PimsnerVoiculescu80}.
\end{proof}

Note that $ K_0 ( C ( X ) ) \cong C ( X , \bZ ) $; we will use this identification throughout the paper. Let $ C ( X , \bZ )^+ $ denote the subset of $ C ( X , \bZ ) $ consisting of $ f $ such that $ f ( x ) \geq 0 $ for all $ x \in X $. Then it is easy to check that $ ( C ( X , \bZ ) , C ( X , \bZ )^+ ) $ is an ordered group and the function $ \chi_X $ is an order unit.

The following is closely related to Proposition 5.1 of \cite{HermanPutnamSkau92}; although the hypotheses are broadened, the proof is essentially the same. Adopting the notation of Proposition \ref{propKTheoryExactSequence}, we denote $ K_0 ( C ( X ) ) / \image ( \id - \alpha_* ) $ by $ K^0 ( X , h ) $.

\begin{prop}
Let $ ( X , h , Z ) $ be a fiberwise essentially minimal zero-dimensional system and adopt the notation of Proposition \ref{propKTheoryExactSequence}. Let $ \pi : C ( X , \bZ ) \to K^0 ( X , h ) $ denote the quotient map. Define $ K^0 ( X , h )^+ = \pi ( C ( X , \bZ )^+ ) $. Then $ ( K^0 ( X , h ) , K^0 ( X , h )^+ , \pi ( 1 ) ) $ is an ordered group with distinguished order unit.
\end{prop}

\begin{proof}
We check the conditions of Definition \ref{defnOrderedGroup}. Conditions (a) and (b) follow from surjectivity of $ \pi $. For condition (c), let $ g \in K^0 ( X , h )^+ \cap -K^0 ( X , h )^+ $. This means that there is $ f_1 \in C ( X , \bZ )^+ $ such that $ \pi ( f_1 ) = g $ and $ f_2 \in C ( X , \bZ )^+ $ such that $ \pi ( - f_2 ) = g $. But then $ \pi ( f_1 + f_2 ) = 0 $ and so $ f_1 + f_2 \in \image ( \id - \alpha_* ) $. Let $ f \in C ( X , \bZ ) $ satisfy $ f - \alpha_* ( f ) = f_1 + f_2 $. Let $ E = f^{ -1 } ( \max_{ x \in X } f ( x ) ) $. Since $ f - \alpha_* ( f ) \geq 0 $, we must have $ h ( E ) \subset E $. Let $ \psi $ be as in Definition \ref{defnFiberwiseEssentiallyMinimal} and let $ z \in \psi ( E ) $ and define $ E_z = E \cap \psi^{ -1 } ( z ) $. Since $ h ( E_z ) \subset E_z $, $ E_z $ is invariant so must intersect the minimal set. But then by Theorem 1.1 of \cite{HermanPutnamSkau92}, since $ E_z \neq \varnothing $, we have $ \bigcup_{ n \in \bZ_{ \geq 0 } } h^n ( E_z ) = \psi^{ -1 } ( z ) $, and so $ E_z = \psi^{ -1 } ( z ) $. Since this holds for all $ z \in \psi ( E ) $, we see that $ f $ is constant on $ \psi^{ -1 } ( z ) $ for all $ z \in Z $. Since $ \psi^{ - 1 } ( z ) $ is invariant for all $ z \in Z $, we must have $ f = \alpha_* ( f ) $, and so $ f_1 + f_2 = 0 $, and since $ f_1 , f_2 \geq 0 $, we see $ f_1 = f_2 = 0 $, and finally we see $ g = 0 $. This proves condition (c).

Finally, the fact that $ \pi ( 1 ) $ is an order unit is also clear from the surjectivity of $ \pi $.
\end{proof}

\begin{thm}\label{thmK0isomK0}
Let $ ( X , h ) $ be a zero-dimensional system. Then adopting the notation of Proposition \ref{propKTheoryExactSequence}, we have $ ( K_0 ( C^* ( \bZ , X , h ) ) , K_0 ( C^* ( \bZ , X , h ) )^+ , 1 ) \cong ( K^0 ( X , h ) , K^0 ( X , h )^+ , 1 ) $.
\end{thm}

\begin{proof}
Since $ \ker ( i_* ) = \image ( \id - \alpha_* ) $, and since $ i : C ( X ) \to C^* ( \bZ , X , h ) $ is the natural inclusion, the map $ i_* $ induces a map $ \vphi : K^0 ( X , h ) \to K_0 ( C^* ( \bZ , X , h ) ) $ which is an isomorphism of groups and satisfies $ K^0 ( X , h )^+ \subset K_0 ( C^* ( \bZ , X , h ) )^+ $. 

Let $ p \in C^* ( \bZ , X , h ) $ be a projection. By applying Corollary \ref{coroCircleAlgebraApproximation} with $ a_1 = p $ and $ \eps = 1/2 $, $ p $ is unitarily equivalent to $ \chi_U $ for some compact open $ U \subset X $. Let $ q $ be the image of $ \chi_U $ under the quotient map $ C ( X ) \to C ( X ) / \image ( \id - \alpha ) $. Then $ [ \vphi ( q ) ] = [ \chi_U ] = [ p ] $. Repeating this argument for $ M_n ( C^* ( \bZ , X , h ) ) $, we see that $ K_0 ( C^* ( \bZ , X , h ) )^+ \subset K^0 ( X , h )^+ $.

Finally, that $ \vphi ( 1 ) = 1 $ is clear, proving the theorem.
\end{proof}

What we have also shown in the previous proof is the following.

\begin{coro}\label{coroSurjectiveOrderedGroupMap}
Let $ ( X , h , Z ) $ be a fiberwise essentially minimal zero-dimensional system. Let $ i : C ( X ) \to C^* ( \bZ , X , h ) $ denote the canonical inclusion. Then the induced map $ i_* : K_0 ( C ( X ) ) \to K_0 ( C^* ( \bZ , X , h ) ) $ is surjective as a map between ordered groups.
\end{coro}

The following definition is from Section 2 of \cite{Poon90}, later studied in the minimal case in \cite{Putnam89}. These have been referred to as ``large subalgebras" in the literature.

\begin{defn}\label{defnAZ}
Let $ ( X , h , Z ) $ be a fiberwise essentially minimal zero-dimensional system. We define $ A_Z $ to be the $ C^* $-algebra generated by $ C ( X ) $ and $ u C_0 ( X \setminus Z ) $.
\end{defn}

The following theorem is contained in Theorem 2.3 of \cite{Poon90}. We provide a direct proof in our context for the reader, which helps give an idea of the AF structure of the large subalgebra.

\begin{thm}\label{thmAZisAF}
Let $ ( X , h , Z ) $ be a fiberwise essentially minimal zero-dimensional system. Then $ A_Z $ is an AF-algebra.
\end{thm}

\begin{proof}
Let $ ( \sP^{ ( n ) } ) $ be a generating sequence of partitions of $ X $. For each $ n \in \bZ_{ > 0 } $, we inductively define systems $ \systemarg{ ( n ) } $ of finite first return time maps. First, let $ \systemarg{ ( 1 ) } $ be any system of finite first return time maps subordinate to $ \sP^{ ( n ) } $ such that $ \sP_1 ( \sS^{ ( 1 ) } ) $ is finer than $ \sP^{ ( 1 ) } $ and such that $ \bigsqcup_{ t = 1 }^{ T^{ ( 1 ) } } X_t^{ ( 1 ) } \supset Z $ (the former is possible by Proposition 1.14 of \cite{Herstedt21} and the latter is possible by Lemma 4.12 of \cite{Herstedt21}). Now, let $ n \in \bZ_{ > 0 } $ and suppose we have chosen $ \systemarg{ ( n ) } $. Let $ \systemarg{ ( n + 1 ) } $ be any system of finite first return time maps subordinate to $ \sP^{ ( n + 1 ) } $ such that $ \sP_1 ( \sS^{ ( n + 1 ) } ) $ is finer than $ \sP^{ ( n + 1 ) } $ and finer than $ \sP_1 ( \sS^{ ( n ) } ) $ and such that $ \bigsqcup_{ t = 1 }^{ T^{ ( n + 1 ) } } X_t^{ ( n + 1 ) } \supset Z $.

Let $ n \in \bZ_{ > 0 } $. Let $ A^{ ( n ) } $ be the finite dimensional $ C^* $-subalgebra of $ C^* ( \bZ , X , h ) $ spanned by the matrix units $ u^{ i - j } \chi_{ h^j ( Y_{ t , k }^{ ( n ) } ) } $ for $ t \in \{ 1 , \ldots , T^{ ( n ) } \} $, $ k \in \{ 1 , \ldots , K_t^{ ( n ) } \} $, and $ i , j \in \{ 0 , \ldots , J_{ t , k }^{ ( n ) } - 1 \} $. We see $ A^{ ( n ) } \cong \bigoplus_{ t = 1 }^{ T^{ ( n ) } } \bigoplus_{ k = 1 }^{ K_t^{ ( n ) } } M_{ J_{ t , k }^{ ( n ) } } $. Notice that $ C ( \sP_1 ( \sS^{ ( n ) } ) ) \subset A^{ ( n ) } $ as the diagonal matrices. Set $ Z^{ ( n ) } = \bigsqcup_{ t = 1 }^{ T^{ ( n + 1 ) } } X_t^{ ( n ) } $ and then notice that $ u C ( X \setminus Z^{ ( n ) } ) \subset A^{ ( n ) } $ as the superdiagonal matrices, so $ A^{ ( n ) } $ is generated by $ C ( \sP_1 ( \sS^{ ( n ) } ) ) $ and $ u C ( X \setminus Z^{ ( n ) } ) $.

Notice that $ A^{ ( n ) } \subset A^{ ( n + 1 ) } $, so we get a directed system of finite dimensional $ C^* $-algebras, whose limit $ A^{ ( \infty ) } $ contains $ C ( X ) $ since $ ( \sP_1 ( \sS^{ ( n ) } ) ) $ is a generating sequence of partitions and since $ \bigcap_{ n = 1 }^\infty Z^{ ( n ) } = Z $, we have $ u C ( X \setminus Z^{ ( n ) } ) \to u C_0 ( X \setminus Z ) \subset A_Z $. It now clear that $ A^{ ( \infty ) } $ is generated by $ C ( X ) $ and $ u C_0 ( X \setminus Z ) $, and is therefore equal to $ A_Z $.
\end{proof}

The following is Lemma 4.2 from \cite{Putnam89}.

\begin{lemma}\label{lemmaProjectionsZeroOnZn}
Adopt the notation of Theorem \ref{thmAZisAF} and its proof. Let $ p $ be a projection in $ C ( X ) \cap A^{ ( n ) } $ and suppose that $ p = 0 $ on $ Z^{ ( n ) } $. Then $ \alpha ( p ) \in C ( X ) \cap A^{ ( n ) } $ and $ [ \alpha ( p ) ] = [ p ] $ in $ K_0 ( A^{ ( n ) } ) $.
\end{lemma}

We finally have the following theorem, which tells us enough about the $ K_0 $ structure of the crossed product to be able to prove Theorems \ref{thmMainTheoremPrelude} and Theorem \ref{thmMainTheorem}. The proof follows that of Theorem 4.1 in \cite{Putnam89}.

\begin{thm}\label{thmK0AZIsomorphism}
Let $ ( X , h , Z ) $ be a fiberwise essentially minimal zero-dimensional system. Then $ K_0 ( A_Z ) \cong K_0 ( C^* ( \bZ , X , h ) ) $ as ordered groups.
\end{thm}

\begin{proof}
Let $ i : A_Z \to C^* ( \bZ , X , h ) $ denote the inclusion map, and let $ i_* : K_0 ( A_Z ) \to K_0 ( C^* ( \bZ , X , h ) ) $ denote map induced by $ i $ on $ K_0 $. Let $ i_1 : C ( X ) \to A_Z $ denote the canonical inclusion, let $ i_2 : C ( X ) \to C^* ( \bZ , X , h ) $ denote the canonical inclusion, and let $ ( i_1 )_* $ and $ ( i_2 )_* $ denote the induced maps on $ K_0 $. We then clearly have $ i \circ i_1 = i_2 $. By Corollary \ref{coroSurjectiveOrderedGroupMap}, $ ( i_2 )_* : K_0 ( C ( X ) ) \to K_0 ( C^* ( \bZ , X , h ) ) $ is a surjective map between ordered groups, and therefore so is $ i^* $.

By Proposition \ref{propKTheoryExactSequence}, $ \ker ( ( i_2 )_* ) = \ran ( \id - \alpha_* ) $. Thus, since $ ( i_2 )_* = ( i_1 )_* \circ i_* $, we have $ ( i_1 )_* ( \ran ( \id - \alpha_* ) ) \subset \ker ( i_* ) $. Now suppose that $ a \in \ker ( i_* ) $. Because $ ( i_1 )_* $ is surjective, we can find $ g \in C ( X , \bZ ) $ such that $ ( i_1 )_* ( g ) = a $. Then $ ( i_2 )_* ( g ) = i_* ( a ) = 0 $, so $ g \in \ker ( ( i_2 )_* ) = \ran ( \id - \alpha_* ) $, so $ a \in ( i_1 )_* ( \ran ( \id - \alpha_* ) ) $. Altogether, we have 
\begin{equation}\label{eq: ran equals ker}
( i_1 )_* ( \ran ( \id - \alpha_* ) ) = \ker ( i_* ) . 
\end{equation}

We adopt the notation of the proof of Theorem \ref{thmAZisAF}. That is, we have a generating sequence of partitions $ ( \sP^{ ( n ) } ) $ and a sequence of systems of finite first return time maps $ ( \sS^{ ( n ) } ) $ and a sequence of subalgebras $ ( A^{ ( n ) } ) $ of $ C^* ( \bZ , X , h ) $ as in the proof. We now claim that $ ( i_1 )_* ( \ran ( \id - \alpha_* ) ) = 0 $. Suppose $ g_1 , g_2 \in C ( X , \bZ ) $ satisfy $ g_1|_Z = g_2|_Z $. Since $ ( \sP^{ ( n ) } ) $ is a generating sequence of partitions and for each $ n \in \bZ_{ > 0 } $, we have $ \sP_1 ( \sS^{ ( n ) } ) $ is finer than $ \sP^{ ( n ) } $, there is some $ n \in \bZ_{ > 0 } $ such that $ g_1 , g_2 \in A^{ ( n ) } $ and such that $ g_1|_{ Z^{ ( n ) } } = g_2|_{ Z^{ ( n ) } } $ (where $ Z^{ ( n ) } $ is defined in the proof of Theorem \ref{thmAZisAF}).  So $ g_1 - g_2 $ is $ 0 $ on $ Z^{ ( n ) } $, so we can write $ g_1 - g_2 $ as a linear combination of projections in $ C ( X ) \cap A^{ ( n ) } $ each of which is zero on $ Z^{ ( n ) } $. So by Lemma \ref{lemmaProjectionsZeroOnZn}, we have $ [ \alpha ( g_1 - g_2 ) ] = [ g_1 - g_2 ] $ in $ K_0 ( A^{ ( n ) } ) $, and so $ ( i_1 )_* ( g_1 - \alpha ( g_1 ) ) = ( i_1 )_* ( g_2 - \alpha ( g_2 ) ) $ in $ K_0 ( A_Z ) $. 

So let $ g \in C ( X , \bZ ) $. Define $ f \in C ( X , \bZ ) $ by $ f = g \circ \psi $. Then $ f|_Z = g|_Z $, and so by the above paragraph, we have $ ( i_1 )_* ( g - \alpha ( g ) ) = ( i_1 )_* ( f - \alpha ( f ) ) $.  But then notice that $ \alpha ( f ) = g \circ \psi \circ h^{ -1 } = g \circ \psi = f $, and so $ ( i_1 )_* ( f - \alpha ( f ) ) = 0 $. Thus, we have $ ( i_1 )_* ( \ran ( \id - \alpha_* ) ) = 0 $. Combining this with (\ref{eq: ran equals ker}), we see $ \ker ( i_* ) = 0 $. Altogether, we have shown that $ i_* $ is an isomorphism of ordered groups.
\end{proof}

\begin{thm}\label{thmK1Formula}
Let $ ( X , h , Z ) $ be a fiberwise essentially minimal zero-dimensional system. Then $ K_1 ( C^* ( \bZ , X , h ) ) \cong C ( Z , \bZ ) $.
\end{thm}

\begin{proof}
Adopt the notation of Proposition \ref{propKTheoryExactSequence}. Then $ K_1 ( C^* ( \bZ , X , h ) ) \cong \ker ( \id - \alpha_* ) $. Identifying $ K_0 ( C ( X ) ) $ with $ C ( X , \bZ ) $, we may replace $ \alpha_* $ with $ \alpha $.

Let $ f \in \ker ( \id - \alpha ) $ and let $ z \in Z $. Suppose $ f|_{ \psi^{ -1 } ( z ) } $ is not constant. Then there is some $ x \in \psi^{ -1 } ( z ) $ such that $ f ( z ) \neq f ( x ) $. Let $ U $ be a compact open subset of $ \psi^{ -1 } ( z ) $ such that $ f ( U ) = f ( z ) $. Since $ ( \psi^{ -1 } ( z ) , h|_{ \psi^{ -1 } ( z ) } ) $ is an essentially minimal zero-dimensional system, by Theorem 1.1 of \cite{HermanPutnamSkau92}, there is an $ n \in \bZ_{ > 0 } $ such that $ x \in h^{ -n } ( U ) $. Let $ x' = h^{ n } ( x ) \in U $. Then $ f ( x' ) \neq f ( x ) = f ( h^{ -n } ( x' ) ) = \alpha^n ( f ) ( x' ) $, and so $ f \neq \alpha^n ( f ) $, and so $ f \neq \alpha ( f ) $, a contradiction to $ x \in \ker ( \id - \alpha ) $. Therefore, $ f|_{ \psi^{ -1 } ( z ) } $ is constant.

Now suppose $ f \in C ( X , \bZ ) $ and suppose $ f|_{ \psi^{ -1 } ( z ) } $ is constant for each $ z \in Z $. Then for each $ x \in X $, we have $ \alpha ( f ) ( x ) = f ( h^{ -1 } ( x ) ) = f ( \psi ( x ) ) = f ( x ) $, and so $ f = \alpha ( f ) $, and so $ f \in \ker ( \id - \alpha ) $. 

Thus, we have 
\[
\ker ( \id - \alpha ) = \{ f \in C ( X , \bZ ) \setdiv \mbox{$f|_{ \psi^{ -1 } ( z ) } $ is constant for each $ z \in Z $} \} \cong C ( Z , \bZ ) 
\]
as desired.
\end{proof}

\begin{coro}
Let $ ( X_1 , h_1 , Z_1 ) $ and $ ( X_2 , h_2 , Z_2 ) $ be fiberwise essentially minimal zero-dimensional systems such that $ C^*( \bZ , X_1 , h_1 ) \cong C^* ( \bZ , X_2 , h_2 ) $. Then $ Z_1 \cong Z_2 $.
\end{coro}

\section{Bratelli Diagrams}

In this section, we explore the construction of ordered Bratelli diagrams associated to fiberwise essentially minimal zero-dimensional systems. This correspondence is used to prove Theorem \ref{thmMainTheoremPrelude}. For work done in the minimal case, see \cite{BezuglyiKwiatkowskiYassawi14} and \cite{GiordanoPutnamSkau95}. For work done in the essentially minimal case, see  \cite{AminiElliottGolestani21} and \cite{HermanPutnamSkau92}.

\begin{defn}
A \emph{Bratteli diagram} $ B $ is a pair of sets $ ( V , E ) $ such that:
\begin{enumerate}[(a)]
\item The set $ V $ is called \emph{the set of vertices of $ B $}. We can write $ V = \bigsqcup_{ n = 0 }^\infty V_n $ where $ V_0 $ contains a single point $ v_0 $ and $ V_n $ is finite and nonempty for all $ n \in \bZ_{ \geq 0 } $. For each $ n \in \bZ_{ \geq 0 } $, we call $ V_n $ \emph{the set of vertices of $ B $ at level $ n $}.
\item The set $ E $ is called \emph{the set of edges of $ B $}. We can write $ E = \bigsqcup_{ n = 1 }^\infty E_n $ where $ E_n $ is finite and nonempty for all $ n \in \bZ_{ > 0 } $. For each $ n \in \bZ_{ > 0 } $, we call $ E_n $ \emph{the set of edges at level $ n $}.
\item There are maps $ r , s : E \to V $ such that for $ n \in \bZ_{ > 0 } $ and $ e \in E_n $, $ s ( e ) \in V_{ n - 1 } $ and $ r ( e ) \in V_n $. Moreover, $ r^{ -1 } ( v ) $ is nonempty for all $ v \in V $ and $ s^{ -1 } ( v ) $ is nonempty for all $ v \in V \setminus V_0 $. The map $ r $ is called the \emph{range map of $ B $} and the map $ s $ is called \emph{the source map of $ B $}.
\end{enumerate}
\end{defn}

\begin{nota}
Let $ B = ( V , E ) $ be a Bratteli diagram. For each $ v \in V $, we denote $ R ( v ) = r ( s^{ -1 } ( v ) ) $, and for each $ v \in V \setminus V_0 $, we denote $ S ( v ) = s ( r^{ -1 } ( v ) ) $. If $ v \in V_n $, then $ R ( v ) $ is the set of all vertices in $ V_{ n + 1 } $ that are connected to $ v $ by an edge, and $ S ( v ) $ is the set of all vertices in $ V_{ n - 1 } $ connected to $ v $ by an edge. In a reasonable sense, this gives us range and source maps for vertices.
\end{nota}

\begin{defn}
An \emph{ordered Bratteli diagram} $ B $ is a Bratteli diagram $ ( V , E ) $ together with a partial order $ \leq $ on $ E $ such that $ e , e' \in E $ are comparable if and only if $ r ( e ) = r ( e' ) $. We write $ B = ( V , E , \leq ) $.
\end{defn}

Let $ B = ( V , E , \leq ) $ be an ordered Bratteli diagram. We define $ E_{ \min } $ ($ E_{ \max } $) to be the set of all edges that are minimal (maximal, resp.) with respect to $ \leq $. We define $ V_{ \min } $ ($ V_{ \max } $) to be the set of all $ v \in V $ such that there is an $ e $ in $ E_{ \min } $ ($ E_{ \max } $, resp.) with $ s ( e ) = v $.

\begin{defn}
Let $ B = ( V , E , \leq ) $ be an ordered Bratteli diagram. We define a \emph{partial Vershik transformation} $ \widetilde{ h }_B : ( X_B \setminus X_{ B ,\max } ) \cup X_{ B ,\min } \to ( X_B \setminus X_{ B , \min } ) \cup X_{ B , \max } $ in the following way. If $ x \in X_{ B , \max } \cap X_{ B , \min } $, we define $ \widetilde{ h }_B ( x ) = x $. If $ x = ( x_1 , x_2 , \ldots ) \in X_B \setminus X_{ B , \max } $, then there is some smallest $ k \in \bZ_{ > 0 } $ such that $ x_k \notin E_{ \max } $. Let $ y_k $ denote the successor of $ x_k $ in $ E $ and let $ ( y_1 , y_2 , \ldots , y_{ k - 1 } ) $ be the unique path from $ v_0 $ to $ s ( y_k ) $ such that $ y_j \in E_{ \min } $ for all $ j \in \{ 1 , \ldots , k - 1 \} $. We define $ \widetilde{ h }_B ( p ) = ( y_1 , y_2 , \ldots , y_k , x_{ k + 1 } , x_{ k + 2 } , \ldots ) $.
\end{defn}

Let $ B = ( V , E , \leq ) $ be an ordered Bratteli diagram. For each $ k , k' \in \bZ_{ > 0 } $ with $ k < k' $, we define $ P_{ k , k' } $ to be the set of all paths from $ V_k $ to $ V_{ k' } $. Formally, $ P_{ k , k' } $ is the set of $ ( e_{ k + 1 } , \ldots , e_{ k' } ) $ such that for all $ j \in \{ k + 1 , \ldots , k' \} $, $ e_j \in E_j $ and for all $ j \in \{ k + 1 , \ldots , k' - 1 \} $, we have $ r ( e_j ) = s ( e_{ j + 1 } ) $.

\begin{defn}
Let $ B = ( V , E , \leq ) $ be an ordered Bratteli diagram. We define the \emph{infinite path space} $ X_B $ to be the set of all sequences $ x = ( x_1 , x_2 , \ldots ) $ where $ x_n \in E_n $ and $ r ( x_n ) = s ( x_{ n + 1 } ) $ for all $ n \in \bZ_{ > 0 } $ together with the topology generated by sets of the form $ U ( e_1 , \ldots , e_k ) $, which is the set of all $ x = ( x_1 , x_2 , \ldots ) $ with $ x_j = e_j $ for all $ j \in \{ 1 , \ldots , k \} $.
\end{defn}

Let $ B = ( V , E , \leq ) $ be an ordered Bratteli diagram. It is easy to see that the infinite path space is a zero-dimensional space. We define $ X_{ B , \min } $ ($ X_{ B , \max }$) to be the set of all $ x = ( x_1 , x_2 , \ldots ) \in X_B $ such that $ x_j $ is in $ E_{ \min } $ ($ E_{ \max } $, resp.) for all $ j \in \bZ_{ > 0 } $.

The following terminology appears in Definition 2.18 of \cite{BezuglyiKwiatkowskiYassawi13}, although we restate it to give more clarity as to when the definition applies.

\begin{defn}
Let $ B = ( V , E , \leq ) $ be an ordered Bratteli diagram and let $ \widetilde{ h }_B $ be its partial Vershik transformation. We say that the ordering on $ B $ is \emph{perfect} if for every $ e \in X_{ B , \min } $, $ \overline{ \orb_{ \widetilde{ h }_B ( e ) } } \cap X_{ B , \max } $ contains a single element, and if for every $ e \in X_{ B , \max } $, $ \overline{ \orb_{ \widetilde{ h }_B ( e ) } } \cap X_{ B , \min } $ contains a single element. In this case, we define the \emph{Vershik transformation} of $ X_B $, denoted by $ h_B $, to be the extension of $ \widetilde{ h }_B $ which, for each $ e \in X_{ B , \min } $, sends the unique element of $ \overline{ \orb_{ \widetilde{ h }_B ( e ) } } \cap X_{ B , \max } $ to $ e $.
\end{defn}

Thus, given an ordered Bratteli diagram $ B = ( V , E , \leq ) $ with a perfect ordering, the system $ ( X_B , h_B ) $ is a zero-dimensional system.


\begin{lemma}\label{lemmaSystemSetUpForDiagrams}
Let $ ( X , h ) $ be a fiberwise essentially minimal zero-dimensional system, let $ \sP $ and $ \sP' $ be partitions of $ X $, and let $ \systembasic $ be a system of finite first return time maps subordinate to $ \sP $ such that for each $ t \in \{ 1 , \ldots , T \} $, we have $ \psi ( X_t ) \subset X_t $. Then there is a system $ \systemarg{ 0 } $ of finite first return time maps subordinate to $ \sP $ and a system $ \systemarg{ \prime } $ of finite first return time maps subordinate to $ \sP' $ such that:
\begin{enumerate}[(a)]
\item The partition $ \sP_1 ( \sS' ) $ is finer than $ \sP' $ and $ \sP_1 ( \sS^0 ) $.
\item The partition $ \sP_1 ( \sS^0 ) $ is finer than $ \sP_1 ( \sS ) $.
\item We have $ T^0 = T $ and for each $ t \in \{ 1 , \ldots , T \} $, we have $ X_t^0 = X_t $.
\item For each $ t' \in \{ 1 , \ldots , T' \} $, there is a $ t \in \{ 1 , \ldots , T^0 \} $ and $ k \in \{ 1 , \ldots , K_t^0 \} $ such that $ X_{ t' }' \subset Y_{ t , k }^0 $.
\item For each $ t' \in \{ 1 , \ldots , T' \} $, there is a $ t \in \{ 1 , \ldots , T^0 \} $ and $ k \in \{ 1 , \ldots , K_t^0 \} $ such that $  X_{ t' }' \subset h^{ J_{ t , k }^0 } ( Y_{ t , k }^0 ) $.
\item For each $ t \in \{ 1 , \ldots , T^0 \} $ and each $ k \in \{ 1 , \ldots , K^0 \} $, there is a $ t' \in \{ 1 , \ldots , T' \} $ such that $ Y_{ t , k }^0 \subset \bigcup_{ j \in \bZ } h^j ( X_{ t' }' ) $.
\end{enumerate}
\end{lemma}

\begin{proof}
We first construct $ \sS' $ and then use it to modify $ \sS $ to obtain $ \sS^0 $. By applying Lemma 4.13 of \cite{Herstedt21}, we may assume that $ \sS $ satisfies the conclusions of the lemma; in particular, for all $ t \in \{ 1 , \ldots , T \} $, we have $ \psi ( X_t ) \subset Y_{ t , 1 } $, and the partitions $ \sP_1 ( \sS ) $ and $ \sP_2 ( \sS ) $ are finer than $ \sP $.  Let $ \sP\dprime $ be a partition finer than $ \sP' $, $ \sP_1 ( \sS ) $, and $ \sP_2 ( \sS ) $. Then apply Lemma 3.2 of \cite{Herstedt21} to obtain a system $ \systemarg{ \prime } $ of finite first return time maps subordinate to $ \sP\dprime $ such that for all $ t' \in \{ 1 , \ldots , T' \} $, there is a $ t \in \{ 1 , \ldots , T \} $ such that $ X_{ t' }' \subset X_t $. Since $ \sP\dprime $ is finer than $ \sP' $, $ \sS' $ is also subordinate to $ \sP' $. By applying Proposition 1.9 of \cite{Herstedt21}, we may assume that $ \sP_1 ( \sS' ) $ is finer than $ \sP\dprime $.

Let $ t' \in \{ 1 , \ldots , T' \} $ and let $ t \in \{ 1 , \ldots , T \} $ be such that $ X_{ t' }' \subset X_t $. Since $ \sP\dprime $ is finer than $ \sP_1 ( \sS ) $, there is some $ k \in \{ 1 , \ldots , K_t \} $ such that 
\begin{equation}\label{eq: X_t' subset Ytk}
X_{ t' }' \subset Y_{ t , k } .
\end{equation}
Since $ \sP\dprime $ is finer than $ \sP_2 ( \sS ) $, there is some $ l \in \{ 1 , \ldots , K_t \} $ such that 
\begin{equation}\label{eq: X_t' subset hJtkYtk}
X_{ t' }' \subset h^{ J_{ t , l } } ( Y_{ t , l } ) .
\end{equation}

Define $ T^0 = T $ and for each $ t \in \{ 1 , \ldots , T \} $, define $ X_t^0 = X_t $ (note that after finishing this construction this verifies conclusion (c) of the lemma). Let $ t \in \{ 1 , \ldots , T \} $ and let $ \{ s ( 1 ) , \ldots , s ( N_t ) \} $ be the set of all $ t' \in \{ 1 , \ldots , T' \} $ such that $ X_{ t' } \subset X_t $. For each $ n \in \{ 1 , \ldots , N_t \} $, let $ \{ a ( n , 1 ) , \ldots , a ( n , C_n ) \} $ be the set of all $ k \in \{ 1 , \ldots , K_t \} $ such that $ Y_{ t , k } \cap \bigcup_{ j \in \bZ } h^j ( X_{ s ( n ) }' ) \neq \varnothing $. Define $ K_t^0 = \sum_{ n = 1 }^{ N_t } C_n $ and define $ C_0 = 0 $. Let $ k \in \{ 1 , \ldots , K_t^0 \} $ and let $ n \in \{ 1 , \ldots , N_t \} $ and $ c \in \{ 1 , \ldots , C_n \} $ be such that $ k = C_{ n - 1 } + c $. Then define 
\begin{equation}\label{eq : Ytk0 defn}
Y_{ t , k }^0 = Y_{ t , a ( n , c ) } \cap \bigcup_{ j \in \bZ } h^j ( X_{ s ( n ) }' ) 
\end{equation}
and define $ J_{ t , k }^0 = J_{ t , a ( n , c ) } $. It is routine to verify that $ \systemarg{ 0 } $ is a system of finite first return time maps subordinate to $ \sP $. By applying Proposition 1.9 of \cite{Herstedt21}, we may assume that $ \sP_1 ( \sS^0 ) $ is finer than $ \sP $ (note that this proves conclusion (b) of the lemma).

We now verify the conclusions of the lemma. Conclusions (a), (b), and (c) have already been verified. Conclusion (d) follows from (\ref{eq: X_t' subset Ytk}) and from the fact that $ \sP_1 ( \sS^0 ) $ is finer than $ \sP_1 ( \sS ) $. Conclusion (e) follows from (\ref{eq: X_t' subset Ytk}) and the fact that that $ \sP_2 ( \sS^0 ) $ is finer than $ \sP_2 ( \sS ) $. Conclusion (f) is shown by (\ref{eq : Ytk0 defn}). This proves the lemma.
\end{proof}

\begin{prop}\label{propAssociatedBratteliDiagram}
Let $ ( X , h , Z ) $ be a fiberwise essentially minimal zero-dimensional system. There is an ordered Bratteli diagram $ B = ( V , E , \leq ) $ with a perfect ordering such that 
\begin{enumerate}[(a)]
\item The system $ ( X_B , h_B , X_{ B , \min } ) $ is conjugate to $ ( X , h , Z ) $.
\item For each $ v $ in $ V_{ \min } $ (or $ V_{ \max } $) there is a $ v' $ in $ V_{ \min } $ ($ V_{ \max } $, resp.) and an edge $ e $ in $ E_{ \min } $ ($ E_{ \max } $, resp.) such that $ s ( e ) = v $ and $ r ( e ) = v' $.
\item For each $ v $ in $ V_{ \min } $ (or $ V_{ \max } $) and each $ e $ in $ E_{ \min } $ ($ E_{ \max } $, resp.) with $ r ( e ) \in R ( v ) $ satisfies $ s ( e ) = v $.
\item For each $ v $ in $ V_{ \min } $ (or $ V_{ \max } $) and each $ m \in \bZ_{ > 0 } $, $ R^m ( v )  = ( R^m \circ S^m \circ R^m ) ( v ) $.
\end{enumerate}
\end{prop}

\begin{proof}
Let $ Z $ and $ \psi $ correspond to $ ( X , h ) $ as in Definition \ref{defnFiberwiseEssentiallyMinimal}, and let $ ( \sP^{ ( n ) } ) $ be a generating sequence of partitions of $ X $. We will construct an ordered Bratteli diagram $ B = ( V , E , \leq ) $ such that $ ( X_B , h_B ) $ is conjugate to $ ( X , h ) $ via a map $ F : X \to X_B $ that satisfies $ F ( Z ) = X_{ B , \min } $. 

First, we construct a sequence $ ( \sS^{ ( n ) } ) $ of finite first return time maps subordinate to $ ( \sP^{ ( n ) } ) $. First, let $ \sS^{ ( 1 ) \prime } $ be any system of finite first return time maps subordinate to $ \sP^{ ( 1 ) } $ such that $ \sP_1 ( \sS^{ ( 1 ) \prime } ) $ is finer than $ \sP^{ ( 1 ) } $ (such a system exists by Proposition 1.9 of \cite{Herstedt21}). We construct the other systems inductively. For each $ n \in \bZ_{ > 0 } $, we apply Lemma \ref{lemmaSystemSetUpForDiagrams} with $ \sS^{ ( n ) \prime } $ in place of $ \sS $, $ \sP^{ ( n ) } $ in place of $ \sP $, $ \sP^{ ( n + 1 ) } $ in place of $ \sP' $, and get $ \sS^{ ( n ) } $ (that is $ \sS^0 $ in the lemma) and $ \sS^{ ( n + 1 )\prime } $ (that is $\sS'$ in the lemma) satisfying the conclusions of the lemma. Thus, to construct the sequence of systems, we only need to define $ \sS^{ ( 0 ) } $ by $ T^{ ( 0 ) } = 1 $ , $ X_1^{ ( 0 ) } = X $, $ K_1^{ ( 0 ) } = 1 $, $ Y_{ 1 , 1 }^{ ( 0 ) } = 1 $, and $ J_{ 1 , 1 }^{ ( 0 ) } = 1 $.

Now we begin to define $ B $. For each $ n \in \bZ_{ \geq 0 } $, define
\[
V_n = \{ ( n , t , k ) \setdiv \mbox{$ t \in \{ 1 , \ldots , T^{ ( n ) } \} $ and $ k \in \{ 1 , \ldots , K_t^{ ( n ) } \} $} \}
\]
The set of edges from $ ( n , t , k ) \in V_n $ to $ ( n + 1 , t' , k' ) \in V_{ n + 1 } $ is the set of all $ ( n + 1 , t' , k' , j ) $ such that $ h^j ( Y_{ t' , k' }^{ ( n + 1 ) } ) \subset Y_{ t , k }^{ ( n ) } $. Note that this is well defined; by assumption, $ h^j ( Y_{ t' , k' }^{ ( n + 1 ) } ) $ is a subset of an element of $ \sP_1 ( \sS^{ ( n ) } ) $, so we do not need to include $ t $ and $ k $ in the tuple defining this edge. We define an order on the edges $ r^{ -1 } ( ( n , t , k ) ) $ by $ ( n , t , k , j_1 ) \leq ( n , t , k , j_2 ) $ if $ j_1 < j_2 $.

We now construct the orbit map $ F : X \to X_B $. Let $ x \in X $. Then for each $ n \in \bZ_{ > 0 } $, there is precisely one $ t \in \{ 1 , \ldots , T^{ ( n ) } \} $, one $ k \in \{ 1 , \ldots , K_t^{ ( n ) } \} $, and one $ j \in \{ 0 , \ldots , J_{ t , k }^{ ( n ) } \} $ such that $ x \in h^j ( Y_{ t , k }^{ ( n ) } ) $. If $ x \in h^j ( Y_{ t , k }^{ ( n ) } ) \cap h^{ j' } ( Y_{ t' , k' }^{ ( n + 1 ) } ) $, then $ j' \geq j $, since otherwise we would have $ h^{ j - j' } ( Y_{ t , k }^{ ( n ) } ) \subset X_t^{ ( n + 1 ) } $ (this follows from Lemma \ref{lemmaSystemSetUpForDiagrams}(c) and (d)), which is not possible since by definition $ h^i ( Y_{ t , k }^{ ( n ) } ) \cap X_t^{ ( n + 1 ) } = \varnothing $ for $ i \in \{ 1 , \ldots , J_{ t , k }^{ ( n ) } - 1 \} $. This, combined with the fact that $ \sP_1 ( \sS^{ ( n + 1 ) } ) $ is finer than $ \sP_1 ( \sS^{ ( n ) } ) $, tells us that $ h^{ j' - j } ( Y_{ t' , k' }^{ ( n + 1 ) } ) \subset Y_{ t , k }^{ ( n ) } $, and therefore there is an edge from $ ( n , t , k ) $ to $ ( n + 1 , t' , k' ) $; namely, $ ( n + 1 , t' , k' , j' - j ) $. Thus, this gives us an infinite path in $ X_B $ associated to $ x $. We define $ F $ by sending $ x $ to this infinite path. 

We now show that $ F $ is injective. Suppose $ x , x' \in X $ and $ F ( x ) = F ( x' ) = ( e_1 , e_2 , \ldots ) $ where we write $ e_n = ( n , t_n , k_n , j_n ) $ for $ n \in \bZ_{ > 0 } $. First, by definition is it clear that for each $ n \in \bZ_{ > 0 } $, there are $ i_n , i_n' \in \{ 0 , \ldots , J_{ t_n , k_n }^{ ( n ) } - 1 \} $ such that $ x \in h^{ i_n } ( Y_{ t_n , k_n }^{ ( n ) } ) $ and $ x' \in h^{ i_n' } ( Y_{ t_n , k_n }^{ ( n ) } ) $. First, notice that $ i_1 = j_1 $ and $ i_1' = j_1 $ by definition of $ F $. Then, by definition of $ F $ we have $ j_2 = i_2 - j_1 $ and $ j_2 = i_2' - j_2 $; in particular, $ i_2 = i_2' $. Proceeding like this, we see that $ i_n = i_n' $ for all $ n \in \bZ_{ > 0 } $. Since $ ( \sP^{ ( n ) } ) $ is a generating sequence of partitions, so is $ ( \sP_1 ( \sS^{ ( n ) } ) ) $, and therefore $ \bigcap_{ n = 0 }^\infty h^{ i_n } ( Y_{ t_n , k_n }^{ ( n ) } ) $ contains at most one element of $ X $. Thus, $ x = x' $, and therefore $ F $ is injective.

Next, we show that $ F $ is surjective. Let $ e = ( e_1 , e_2 , \ldots ) \in X_B $ and write $ e_n = ( n , t_n , k_n , j_n ) $ for $ n \in \bZ_{ > 0 } $. We construct a sequence $ ( i_n ) $ with $ i_n \in \{ 0 , \ldots , J_{ t_n , k_n }^{ ( n ) } - 1 \} $ for all $ n \in \bZ_{ > 0 } $ such that $ \bigcap_{ n = 0 }^\infty h^{ i_n } ( Y_{ t_n , k_n }^{ ( n ) } ) $ is nonempty and contains the element of $ X $ that $ F $ maps to $ e $. First, let $ i_1 = j_1 $. Then, for all $ n \in \bZ_{ > 1 } $, let $ i_n = j_n + i_{ n - 1 } $ (note that this can be rewritten as $ i_n = \sum_{ k = 1 }^n j_k $). The claim now follows from the definition of $ B $, and $ F $ is therefore surjective.


So far we have shown that $ F $ is bijective. We now show that $ F $ is a homeomorphism. Let $ U ( e_1 , \ldots , e_N ) $ be an element of the basis of the topology of $ X_B $. For each $ n \in \{ 1 , \ldots , N \} $, write $ e_n = ( n , t_n , k_n , j_n ) $. For each $ n \in \{ 1 , \ldots , N \} $, write $ j_n' = \sum_{ k = 1 }^n j_k $. We claim that if $ x \in h^{ j_N' } ( Y_{ t_N , k_N }^{ ( N ) } ) $, then $ F ( x ) \in U ( e_1 , \ldots , e_N ) $. So let $ x \in h^{ j_N' } ( Y_{ t_N , k_N }^{ ( N ) } ) $. First notice that by definition of $ e_N $, 
\[
h^{ j_N } ( Y_{ t_N , k_N }^{ ( N ) } ) \subset Y_{ t_{ N - 1 } , k_{ N - 1 } }^{ ( N - 1 ) } .
\]
Thus,
\[
h^{ j_N' } ( Y_{ t_N , k_N }^{ ( N ) } ) \subset h^{ j_{ N - 1 }' } ( Y_{ t_{ N - 1 } , k_{ N - 1 } }^{ ( N - 1 ) } ) ,
\]
since $ j_N + j_{ N - 1 }' = j_N' $. Similarly, we have 
\[
h^{ j_{ N - 1 } } ( Y_{ t_{ N - 1 } , k_{ N - 1 } }^{ ( N - 1 ) } ) \subset Y_{ t_{ N - 2 } , k_{ N - 2 } }^{ ( N - 2 ) } .
\]
Thus, for every $ n \in \{ 1 , \ldots , N \} $, we have 
\[
x \in h^{ j_n' } ( Y_{ t_n , k_n }^{ ( N ) } ) ,
\]
and so since $ j_n' - j_{ n - 1 }' = j_n $, the $ n $th edge of $ F ( x ) $ is indeed $ e_n $, and $ F ( x ) \in U ( e_1 , \ldots , e_N ) $ as desired. Next, we claim that if $ x \in X $ satisfies $ F ( x ) \in U ( e_1 , \ldots , e_N ) $, then $ x \in h^{ j_N' } ( Y_{ t_N , k_N }^{ ( N ) } ) $. So let $ x \in X $ satisfy $ F ( x )  \in U ( e_1 , \ldots , e_N ) $. Then for each $ n \in \{ 1 , \ldots , N \} $, $ x \in h^{ i_n } ( J_{ t_n , k_n }^{ ( n ) } ) $ for some $ i_n \in \{ 0 , \ldots , J_{ t_n , k_n }^{ ( n ) } - 1 \} $. It is clear that $ i_1 = j_1 $. Then, notice that $ i_2 $ is such that $ j_2 = i_2 - j_1 $, and so $ i_2 = j_1 + j_2 = j_2' $. Repeating this process inductively, we see that $ i_N = j_N' $, and so $ x \in h^{ j_N' } ( Y_{ t_N , k_N }^{ ( N ) } ) $ as desired. Altogether, this shows that $ F $ is a homeomorphism.

If $ x \in Z $, then there are sequences of integers $ ( t_n ) $ and $ ( k_n ) $ such that $ x \in \bigcap_{ n = 0 }^\infty Y_{ t_n , k_n }^{ ( n ) } $. By definition of the order on $ B $, this means that $ F ( x ) \in X_{ B , \min } $. Conversely, suppose $ x \in X $ satisfies $ F ( x ) \in X_{ B , \min } $. Write $ F ( x ) = ( e_1 , e_2 , \ldots ) $ and for $ n \in \bZ_{ > 0 } $ write $ e_n = ( n , t_n , k_n , i_n ) $. Since $ e_n $ is minimal, $ i_n $ is the minimal element of $ \{ 0 , \ldots , J_{ t_n , k_n }^{ ( n ) } - 1 \} $ such that $ h^{ i_n } ( Y_{ t_n , k_n }^{ ( n ) } ) \subset X_{ t_{ n - 1 } }^{ ( n - 1 ) } $. But since $ X_{ t_n }^{ ( n ) } \subset X_{ t_{ n - 1 } }^{ ( n - 1 ) } $, we have $ i_n = 0 $. Hence, $ x \in \bigcap_{ n = 0 }^\infty Y_{ t_n , k_n }^{ ( n ) } $. Thus, $ F ( Z ) = X_{ B , \min } $. Also notice that if $ x \in h^{ -1 } ( Z ) $, there are sequences of integers $ ( t_n ) $ and $ ( k_n ) $ such that $ x \in \bigcap_{ n = 0 }^\infty h^{ J_{ t_n , k_n }^{ ( n ) } - 1 } ( Y_{ t_n , k_n }^{ ( n ) } ) $. By definition of the order on $ B $, this means that $ F ( x ) \in X_{ B , \max } $. Similarly, the converse holds, and so $ F ( h^{ - 1 } ( Z ) ) = X_{ B , \max } $.

We now show that $ ( F \circ h )|_{ X \setminus h^{ -1 } ( Z ) } = ( \widetilde{ h }_B \circ F )|_{ X \setminus h^{ -1 } ( Z ) } $. Let $ x \in X \setminus h^{ -1 } ( Z ) $. For each $ n \in \bZ_{ > 0 } $, let $ t_n \in \{ 1 , \ldots , T^{ ( n ) } \} $, let $ k_n \in \{ 1 , \ldots , K_{ t_n }^{ ( n ) } \} $, and let $ j_n \in \{ 0 , \ldots , K_{ t_n }^{ ( n ) } - 1 \} $ satisfy $ x \in h^{ j_n } ( Y_{ t_n , k_n }^{ ( n ) } ) $. Since $ x \notin h^{ -1 } ( Z ) $, there is some smallest $ N \in \bZ_{ > 0 } $ such that $ j_N \neq J_{ t_N , k_N }^{ ( N ) } - 1 $. We have 
\[
h^{ j_N + 1 } ( Y_{ t_N , k_N }^{ ( N ) } ) \subset J_{ t_{ N - 1 } , k_{ N - 1 } }^{ ( N - 1 ) } ( Y_{ t_{ N - 1 } , k_{ N - 1 } }^{ ( N - 1 ) } )
\]
and so by Lemma \ref{lemmaSystemSetUpForDiagrams}(d), there is a $ k_{ N - 1 }' \in \{ 1 , \ldots , K_{ t_{ N - 1 } }^{ ( N - 1 ) } \} $ such that 
\begin{equation}\label{eq: hJN+1 equation}
h^{ j_N + 1 } ( Y_{ t_N , k_N }^{ ( N ) } ) \subset Y_{ t_{ N - 1 } , k_{ N - 1 }' }^{ ( N - 1 ) } .
\end{equation} 
Inductively, we can find for each $ n \in \{ 1 , \ldots , N - 2 \} $ a $ k_n' \in \{ 1 , \ldots , K_{ t_n }^{ ( n ) } \} $ such that $ Y_{ t_{ n + 1 }' , k_{ n + 1 }' }^{ ( n + 1 ) } \subset Y_{ t_{ n }' , k_n' }^{ ( n ) } $. Since $ j_N + 1 < J_{ t_N , k_N }^{ ( N ) } $ and since for each $ n \in \bZ_{ > N } $ we have $ X_{ t_n }^{ ( n ) } \subset X_{ t_N }^{ ( N ) } $, it follows that $ j_n + 1 < J_{ t_N , k_N }^{ ( N ) } $ as well. So let $ k_n' = k_n $ and let $ j_n' = 0 $ for $ n \in \{ 1 , \ldots , N - 1 \} $, and let $ j_n' = j_n + 1 $ for $ n \in \bZ_{ \geq N } $. Then for each $ n \in \bZ_{ > 0 } $, $ h ( x ) \in h^{ j_n' } ( Y_{ t_n , k_n' }^{ ( n ) } ) $. 

Now write $ F ( x ) = ( e_1 , e_2 , \ldots ) $ and $ \widetilde{ h }_B ( F ( x ) ) = ( e_1' , e_2' , \ldots ) $ and for each $ n \in \bZ_{ > 0 } $ write $ e_n = ( n , s_n , l_n , i_n ) $ and $ e_n' = ( n , s_n' , l_n' , i_n' ) $. By definition of $ F $, for all $ n \in \bZ_{ > 0 } $ we have $ s_n = t_n $, $ l_n = k_n $, and $ i_n = j_n - j_{ n - 1 } $ where $ j_0 = 0 $. We also see that $ N $ is the smallest element of $ \bZ_{ > 0 } $ such that $ e_N \notin E_{ \max } $, so $ e_N' $ is the successor of $ e_N $, $ ( e_1', \ldots , e_{ N - 1 }' ) $ is the minimal path such that $ r ( e_{ N - 1 }' ) = s ( e_N' ) $, and $ e_n' = e_n $ for all $ n \in \bZ_{ > N } $. In particular we see that $ s_n' = t_n $ for all $ n \in \bZ_{ > 0 } $, $ l_n' = k_n $ for all $ n \in \bZ_{ \geq N } $, $ i_n' = 0 $ for $ n \in \{ 1 , \ldots , N - 1 \} $ and $ i_n' = i_n = j_n - j_{ n - 1 } $ for all $ n \in \bZ_{ > N } $. Observe that $ i_N' $ is the smallest integer greater than $ i_N $ such that $ h^{ i_N' } ( Y_{ t_N , k_N }^{ ( N ) } ) \subset Y_{ t_{ N - 1 } , l_{ N - 1 }' }^{ ( N - 1 ) } $. Thus, this combined with 
\begin{equation}\label{eq: hiN equation}
h^{ i_N } ( Y_{ t_N , k_N }^{ ( N ) } ) \subset Y_{ t_{ N - 1 } , k_{ N - 1 } }^{ ( N - 1 ) } 
\end{equation}
tells us 
\begin{align*}
i_N' &= i_N + J_{ t_{ N - 1 } , k_{ N - 1 } }^{ ( N - 1 ) } \\
&= j_N - j_{ N - 1 } + J_{ t_{ N - 1 } , k_{ N - 1 } }^{ ( N - 1 ) } \\
&= j_N - ( J_{ t_{ N - 1 } , k_{ N - 1 } }^{ ( N - 1 ) } - 1 ) + J_{ t_{ N - 1 } , k_{ N - 1 } }^{ ( N - 1 ) } \\
&= j_N + 1 \\
&= j_N' .
\end{align*}
From (\ref{eq: hJN+1 equation}) and (\ref{eq: hiN equation}) we see $ l_{ N - 1 }' = k_{ N - 1 }' $. Similarly, for $ n \in \{ 2 , \ldots , N - 1 \} $ since $ i_n' = j_n' = 0 $, we have $ l_{ n - 1 }' = k_{ n - 1 }' $. Write $ F ( h ( x ) ) = ( e_1\dprime , e_2\dprime , \ldots ) $ and for each $ n \in \bZ_{ > 0 } $ write $ e_n\dprime = ( n , s_n\dprime , l_n\dprime , i_n\dprime ) $. For each $ n \in \bZ_{ > 0 } $, we have $ s_n\dprime = t_n' = t_n $ and $ l_n\dprime = k_n' $. For $ n \in \{ 1 , \ldots , N - 1 \} $ , we have $ i_n\dprime = i_n' = 0 $. We also have $ i_N\dprime = j_N' - j_{ N - 1 }' = j_N' - 0 = i_N' $. For $ n \in \bZ_{ > N } $, we have 
\begin{align*}
i_n\dprime = j_n' - j_{ n - 1 }' &= j_n + 1 - ( j_{ n - 1 } - 1 ) \\
&= j_n - j_{ n - 1 } \\
&= i_n \\
&= i_n' .
\end{align*}
Altogether, we see $ F ( h ( x ) ) = \widetilde{ h }_B ( F ( x ) ) $, and so 
\begin{equation}\label{eq: partial orbit map}
( F \circ h )|_{ X \setminus h^{ -1 } ( Z ) } = ( \widetilde{ h }_B \circ F )|_{ X \setminus h^{ -1 } ( Z ) } .
\end{equation}

We now show that the order on $ B $ is perfect. For each $ x \in X $, $ z $ is in the minimal set of the essentially minimal zero-dimensional system $ ( \psi^{ -1 } ( Z ) , h|_{ \psi^{ -1 } ( Z ) } ) $, $ \overline{ \orb } ( x ) $ contains exactly one element of $ Z $ and exactly one element of $ h^{ -1 } ( Z ) $. Now, (\ref{eq: partial orbit map}) combined with $ F ( Z ) = X_{ B , \min } $ and $ F ( h^{ -1 } ( Z ) ) = X_{ B , \max } $ tells us that the ordering on $ B $ is perfect. It is now clear that $ F \circ h = h_B \circ F $. This proves conclusion (a) of the proposition.

Before we prove the rest, we first prove two claims that will be used a few times in the proof.

Claim ($\ast$): Let $ n \in \bZ_{ > 0 } $, $ t_n \in \{ 1 , \ldots , T^{ ( n ) } \} $, $ k_n \in \{ 1 , \ldots , K_{ t_n }^{ ( n ) } \} $, and $ t_{ n + 1 } \in \{ 1 , \ldots , T^{ ( n + 1 ) } \} $. If $ X_{ t_{ n + 1 } }^{ ( n + 1 ) } \subset h^{ J_{ t_n , k_n }^{ ( n ) } } ( Y_{ t_n , k_n }^{ ( n ) } ) $, then for any $ k_{ n + 1 } \in \{ 1 , \ldots , K_{ t_{ n + 1 } }^{ ( n + 1 ) } \} $, there is an $ i_{ n + 1 } \in \{ 1 , \ldots , J_{ t_{ n + 1 } , k_{ n + 1 } }^{ ( n + 1 ) } \} $ such that $ e = ( n + 1 , t_{ n + 1 } , k_{ n + 1 } , i_{ n + 1 } ) $ is a maximal edge with $ s ( e ) = ( n , t_n , k_n ) $.

We now prove ($\ast$). Let $ k_{ n + 1 } \in \{ 1 , \ldots , K_{ t_{ n + 1 } }^{ ( n + 1 ) } \} $. Then 
\[
h^{ J_{ t_{ n + 1 } , k_{ n + 1 } }^{ ( n + 1 ) } } ( Y_{ t_{ n + 1 } , k_{ n + 1 } } ) \subset X_{ t_{ n + 1 }^{ ( n + 1 ) } } \subset h^{ J_{ t_n , k_n }^{ ( n ) } } ( Y_{ t_n , k_n }^{ ( n ) } ),
\]
so we have
\[
h^{ J_{ t_{ n + 1 } , k_{ n + 1 } }^{ ( n + 1 ) } - J_{ t_n , k_n }^{ ( n ) } } ( Y_{ t_{ n + 1 } , k_{ n + 1 } } ) \subset Y_{ t_n , k_n }^{ ( n ) } 
\]
Set $ i_{ n + 1 } = J_{ t_{ n + 1 } , k_{ n + 1 } }^{ ( n + 1 ) } - J_{ t_n , k_n }^{ ( n ) } $. Let $ j \in \{ 1 , \ldots , J_{ t_n , k_n }^{ ( n ) } - 1 \} $. Then 
\[
h^{ i_{ n + 1 } + j } ( Y_{ t_{ n + 1 } , k_{ n + 1 } }^{ ( n + 1 ) } ) \subset h^j ( Y_{ t_n , k_n }^{ ( n ) } ),
\]
and since $ h^j ( Y_{ t_n , k_n }^{ ( n ) } ) \cap \left( \bigsqcup_{ t = 1 }^{ T^{ ( n ) } } X_t^{ ( n ) } \right) = \varnothing $, we indeed see that $ e = ( n + 1 , t_{ n + 1 } , k_{ n + 1 } , i_{ n + 1 } ) \in E_{ \max } $. This proves ($\ast$).

Claim ($\ast \ast$): If there is a $ k_{ n + 1 } \in \{ 1 , \ldots , K_{ t_{ n + 1 } }^{ ( n + 1 ) } \} $ and a $ i \in \{ 1 , \ldots , J_{ t_{ n + 1 } , k_{ n + 1 } }^{ ( n + 1 ) } \} $ such that $ e = ( n + 1 , t_{ n + 1 } , k_{ n + 1 } , i_{ n + 1 } ) $ is a maximal edge with $ s ( e ) = ( n , t_n , k_n ) $, then $ X_{ t_{ n + 1 } }^{ ( n + 1 ) } \subset h^{ J_{ t_n , k_n }^{ ( n ) } } ( Y_{ t_n , k_n }^{ ( n ) } ) $. 

We now prove ($\ast \ast$). Since $ s ( e ) = ( n , t_n , k_n ) $, we have
\[
h^{ i_{ n + 1 } } ( Y_{ t_{ n + 1 }, k_{ n + 1 } }^{ ( n + 1 ) } ) \subset Y_{ t_n , k_n }^{ ( n ) }
\]
Since $ e $ is maximal, there is no $ j \in \{ i_{ n + 1 } + 1 , \ldots , J_{ t_{ n + 1 } , k_{ n + 1 } }^{ ( n + 1 ) } - 1 \} $ with $ h^j ( Y_{ t_{ n + 1 }, k_{ n + 1 } }^{ ( n + 1 ) } ) \subset X_{ t_n , k_n }^{ ( n ) } $. However, notice that 
\[
h^{ i_{ n + 1 } + J_{ t_n , k_n }^{ ( n ) } } ( Y_{ t_{ n + 1 } , k_{ n + 1 } }^{ ( n + 1 ) } ) \subset h^{ J_{ t_n , k_n }^{ ( n ) } } (  Y_{ t_n , k_n }^{ ( n ) } ) \subset X_{ t_n , k_n }^{ ( n ) } .
\]
Thus, we must have $ i_{ n + 1 } + J_{ t_n , k_n }^{ ( n ) } = J_{ t_{ n + 1 } , k_{ n + 1 } }^{ ( n + 1 ) } $. Thus, $ X_{ t_{ n + 1 } }^{ ( n + 1 ) } \cap h^{ J_{ t_n , k_n }^{ ( n ) } } (  Y_{ t_n , k_n }^{ ( n ) } )  \neq \varnothing $, and so by Lemma \ref{lemmaSystemSetUpForDiagrams}(e), we actually have $ X_{ t_{ n + 1 } }^{ ( n + 1 ) } \subset h^{ J_{ t_n , k_n }^{ ( n ) } } (  Y_{ t_n , k_n }^{ ( n ) } ) $. This proves ($\ast \ast$).

We now prove conclusion (b) of the proposition. Let $ v \in V_{ \min } $. Write $ v = ( n , t_n , k_n ) $. Since $ v \in V_{ \min } $, there is some $ v\dprime \in V_{ n + 1 } $ and some $ e' \in E_{ \min } $ with $ s ( e' ) = v $ and $ r ( e' ) = v\dprime $. Write $ v\dprime = ( n + 1 , t_{ n + 1 } , k_{ n + 1 } ) $ and then $ e' = ( n + 1 , t_{ n + 1 } , k_{ n + 1 } , i_{ n + 1 } ) $. By Lemma \ref{lemmaSystemSetUpForDiagrams}(d), since $ e' \in E_{ \min } $, we have $ i_{ n + 1 } = 0 $, so $ Y_{ t_{ n + 1 } , k_{ n + 1 } }^{ ( n + 1 ) } \subset Y_{ t_n , k_n }^{ ( n ) } $. Again by Lemma \ref{lemmaSystemSetUpForDiagrams}(d), we have 
\begin{equation}\label{eq: whole base is a subset}
X_{ t_{ n + 1 } }^{ ( n + 1 ) } \subset Y_{ t_n , k_n }^{ ( n ) } .
\end{equation}
Now, let $ t_{ n + 2 } \in \{ 1 , \ldots , T^{ ( n + 2 ) } \} $ satisfy 
\[
X_{ t_{ n + 2 } }^{ ( n + 2 ) } \cap \left( \bigcup_{ j \in \bZ } h^j ( X_{ t_{ n + 1 } }^{ ( n + 1 ) } ) \right) \neq \varnothing .
\]
By Lemma \ref{lemmaSystemSetUpForDiagrams}(d), there is actually a $ k_{ n + 1 }' \in \{ 1 , \ldots , K_{ t_{ n + 1 } }^{ ( n + 1 ) } \} $ such that
\[
X_{ t_{ n + 2 } }^{ ( n + 2 ) } \subset Y_{ t_{ n + 1 } , k_{ n + 1 }' }^{ ( n + 1 ) }.
\]
So for any $ k_{ n + 2 } \in \{ 1 , \ldots , K_{ t_{ n + 2 } }^{ ( n + 2 ) } \} $, we have a minimal edge $ e\dprime = ( n + 2 , t_{ n + 2 } , k_{ n + 2 } , 0 ) $ with $ s ( e\dprime ) = v' = ( n + 1 , t_{ n + 1 } , k_{ n + 1 }' ) $. Thus, $ v' \in V_{ \min } $. By (\ref{eq: whole base is a subset}), $ Y_{ t_{ n + 1 } , k_{ n + 1 }' }^{ ( n + 1 ) } \subset Y_{ t_n , k_n }^{ ( n ) } $, and so there is a minimal edge $ e = ( n + 1 , t_{ n + 1 } , k_{ n + 1 }' , 0 ) $ with $ s ( e ) = v $ and $ r ( e ) = v' $. 

Now let $ v \in V_{ \max } $ and write $ v = ( n , t_n , k_n ) $. Since $ v \in V_{ \max } $, there is some $ v\dprime \in V_{ n + 1 } $ and some $ e' \in E_{ \max } $ with $ s ( e' ) = V $ and $ r ( e' ) = v\dprime $. Write $ v\dprime = ( n + 1 , t_{ n + 1 } , k_{ n + 1 } ) $ and then $ e' = ( n + 1 , t_{ n + 1 } , k_{ n + 1 } , i_{ n + 1 } ) $. By $(\ast \ast)$, we have 
\[
X_{ t_{ n + 1 } }^{ ( n + 1 ) } \subset h^{ J_{ t_n , k_n }^{ ( n ) } } ( Y_{ t_n , k_n }^{ ( n ) } )
\]
Now, let $ t_{ n + 2 } \in \{ 1 , \ldots , T^{ ( n + 2 ) } \} $ satisfy 
\[
X_{ t_{ n + 2 } }^{ ( n + 2 ) } \cap \left( \bigcup_{ j \in \bZ } h^j ( X_{ t_{ n + 1 } }^{ ( n + 1 ) } ) \right) \neq \varnothing .
\]
By Lemma \ref{lemmaSystemSetUpForDiagrams}(e), there is actually a $ k_{ n + 1 }' \in \{ 1 , \ldots , K_{ t_{ n + 1 } }^{ ( n + 1 ) } \} $ such that
\[
X_{ t_{ n + 2 } }^{ ( n + 2 ) } \subset h^{ J_{ t_{ n + 1 } , k_{ n + 1 }' }^{ ( n + 1 ) } } ( Y_{ t_{ n + 1 } , k_{ n + 1 }' }^{ ( n + 1 ) } ) .
\]
By $(\ast)$, $ v' = ( n + 1 , t_{ n + 1 } , k_{ n + 1 }' ) \in V_{ \max } $. By $(\ast)$, there is $ e \in E_{ \max } $ with $ s ( e ) = v $ and $ r ( e ) = v' $. This completes the proof of conclusion (b).

We now prove conclusion (c) of the proposition. Let $ v \in V_{ \min } $ and $ e \in E_{ \min } $ satisfy $ r ( e ) \in R ( v ) $. Write $ v = ( n , t_n , k_n ) $ and $ e = ( n + 1 , t_{ n + 1 } , k_{ n + 1 } , i_{ n + 1 } ) $. Since $ r ( e ) \in R ( v ) $, there is an edge $ e' = ( n + 1 , t_{ n + 1 } , k_{ n + 1 } , i_{ n + 1 }' ) $ with $ s ( e' ) = v $, which tells us there is $ j_{ n + 1 } \in \{ 0 , \ldots , J_{ t_{ n + 1 } , k_{ n + 1 } }^{ ( n + 1 ) } - 1 \} $ such that $ h^{ j_{ n + 1 } } ( Y_{ t_{ n + 1 } , k_{ n + 1 } } ) \subset Y_{ t_n , k_n }^{ ( n ) } $. By Lemma \ref{lemmaSystemSetUpForDiagrams}(f), $ Y_{ t_n , k_n }^{ ( n ) } \subset \bigcup_{ j \in \bZ } h^j ( X_{ t_{ n + 1 } }^{ ( n + 1 ) } ) $. Since $ v \in V_{ \min } $, there is some edge $ e\dprime = ( n + 1 , t_{ n + 1 } , k_{ n + 1 }\dprime , i_{ n + 1 }\dprime ) \in E_{ \min } $ with $ s ( e ) = v $. By Lemma \ref{lemmaSystemSetUpForDiagrams}(d), since $ e\dprime \in E_{ \min } $, we have $ i_{ n + 1 }\dprime = 0 $, so $ Y_{ t_{ n + 1 } , k_{ n + 1 }\dprime }^{ ( n + 1 ) } \subset Y_{ t_n , k_n }^{ ( n ) } $. Again by Lemma \ref{lemmaSystemSetUpForDiagrams}(d), we actually have 
\begin{equation}\label{eq: Xtn+1 subset Ytn}
X_{ t_{ n + 1 } }^{ ( n + 1 ) } \subset Y_{ t_n , k_n }^{ ( n ) } .
\end{equation}
Now, since $ e \in E_{ \min } $, by Lemma \ref{lemmaSystemSetUpForDiagrams}(d), we have $ i_{ n + 1 } = 0 $. But by (\ref{eq: Xtn+1 subset Ytn}), we must have $ Y_{ t_{ n + 1 } , k_{ n + 1 } }^{ ( n + 1 ) } \subset  Y_{ t_n , k_n }^{ ( n ) } $, which means that $ s ( e ) = v $. 

Let $ v \in V_{ \max } $ and $ e \in E_{ \max } $ satisfy $ r ( e ) \in R ( v ) $. Write $ v = ( n , t_n , k_n ) $ and $ e = ( n + 1 , t_{ n + 1 } , k_{ n + 1 } , i_{ n + 1 } ) $. Since $ r ( e ) \in R ( v ) $, there is an edge $ e' = ( n + 1 , t_{ n + 1 } , k_{ n + 1 } , i_{ n + 1 }' ) $ with $ s ( e' ) = v $, which tells us there is $ j_{ n + 1 } \in \{ 0 , \ldots , J_{ t_{ n + 1 } , k_{ n + 1 } }^{ ( n + 1 ) } - 1 \} $ such that $ h^{ j_{ n + 1 } } ( Y_{ t_{ n + 1 } , k_{ n + 1 } }^{ ( n + 1 ) } ) \subset Y_{ t_n , k_n }^{ ( n ) } $. By Lemma \ref{lemmaSystemSetUpForDiagrams}(f), $ Y_{ t_n , k_n }^{ ( n ) } \subset \bigcup_{ j \in \bZ } h^j ( X_{ t_{ n + 1 } }^{ ( n + 1 ) } ) $. Since $ v \in V_{ \max } $, there is some edge $ e\dprime = ( n + 1 , t_{ n + 1 } , k_{ n + 1 }\dprime , i_{ n + 1 }\dprime ) \in E_{ \max } $ with $ s ( e ) = v $. By $(\ast \ast)$, we have 
\[
X_{ t_{ n + 1 } }^{ ( n + 1 ) } \subset h^{ J_{ t_n , k_n }^{ ( n ) } } ( Y_{ t_n , k_n }^{ ( n ) } ) .
\]
By ($\ast$), $ s ( e ) = v $. This completes the proof of conclusion (c).

Now we prove conclusion (d) of this lemma. Let $ v \in V_{ \min } $, let $ m \in \bZ_{ > 0 } $, and write $ v = ( n , t_n , k_n ) $. It is clear that $ v \in ( S^m \circ R^m ) ( v ) $, and so 
\[
R^m ( v ) \subset ( R^m \circ S^m \circ R^m ) ( v ) .
\]
Now, let $ w = ( n + m , t_{ n + m } , k_{ n + m } ) \in R^m ( v ) $. By Lemma \ref{lemmaSystemSetUpForDiagrams}(d) (applied $ m $ times), we have 
\[
\bigcup_{ j \in \bZ } h^j ( X_{ t_{ n + m } }^{ ( n + m ) } ) \subset \bigcup_{ j \in \bZ } h^j ( X_{ t_n }^{ ( n ) } )
\]
So every element of $ ( S^m \circ R^m ) ( v ) $ has the form $ ( n , t_n , k_n' ) $ for some $ k_n' \in \{ 1 , \ldots , K_{ t_n }^{ ( n ) } \} $. Let $ w' = ( n + m , t_{ n + m }' , k_{ n + m }' )  \in ( R^m \circ S^m \circ R^m ) ( v ) $, meaning there is an path $ ( e_1 , \ldots , e_m ) $ with $ s ( e_1 ) = ( n , t_n , k_n' ) $ for some $ k_n' \in \{ 1 , \ldots , K_{ t_n }^{ ( n ) } \} $ and $ r ( e_m ) = w' $. Write $ e_m = ( n + m , t_{ n + m }' , k_{ n + m }' , i_{ n + m }' ) $. This means that $ h^{ i_{ n + m }' } ( Y_{ t_{ n + m }' , k_{ n + m }' }^{ ( n + m ) } ) \subset Y_{ t_n , k_n' }^{ ( n ) } $. But then by Lemma \ref{lemmaSystemSetUpForDiagrams}(d) (applied $ m $ times), we must have $ Y_{ t_{ n + m }' , k_{ n + m }' }^{ ( n + m ) } \subset Y_{ t_n , k_n }^{ ( n ) } $. Therefore $ w' \in R^m ( v ) $ (by a minimal path), as desired. An identical argument using Lemma \ref{lemmaSystemSetUpForDiagrams}(e) in place of (d) shows that this equation also holds when $ v \in V_{ \max } $. This proves conclusion (d) of this proposition and therefore finishes the proof of the proposition.
\end{proof}

\begin{prop}\label{propTelescopingPreservesAssociatedDiagramProperties}
Let $ ( X , h ) $ be a fiberwise essentially minimal zero-dimensional system and let $ B = ( V , E , \leq ) $ be an ordered Bratteli diagram that satisfies the conclusions of Proposition \ref{propAssociatedBratteliDiagram}. If $ B' = ( V' , E' , \leq' ) $ is a telescoping of $ B $, then $ B' $ also satisfies the conclusions of Proposition \ref{propAssociatedBratteliDiagram}.
\end{prop}

\begin{proof}
Let $ ( k_n ) $ be the telescoping sequence corresponding to $ B' $, so that $ V_n' = V_{ k_n } $ and $ E_n' = P_{ k_n + 1 , k_{ n + 1 } } $ for all $ n \in \bZ_{ > 0 } $ (where $ k_0 = 0 $). This identification induces a map $ \vphi : X_B \to X_{ B' } $ by sending $ e = ( e_1 , e_2 , \ldots ) $ to $ \vphi ( e ) = ( ( e_1 , \ldots , e_{ k_1 } ) , ( e_{ k_1 + 1 } , \ldots , e_{ k_2 } ) , \ldots ) $. By definition of the induced order on a telescoped Bratteli diagram, $ \vphi $ is a conjugation, and therefore $ B' $ satisfies conclusion (a) of Proposition \ref{propAssociatedBratteliDiagram}.

We now show that conclusion (b) holds. Let $ v \in V_{ n , \min }' $. What we are looking for is a $ v' \in V_{ n + 1 , \min }' $ and $ e \in E_{ n + 1 , \min }' $ such that $ s ( e ) = v $ and $ r ( e ) = v' $. We can regard all of this as happening in $ B $ instead of $ B' $, so that we have $ v \in V_{ k_n , \min } $, and we are looking for $ v' \in V_{ k_{ n + 1 } , \min } $ and a minimal path $ ( e_{ k_n + 1 } , \ldots , e_{ k_{ n + 1 } } ) \in P_{ k_n + 1 , k_{ n + 1 } } $. By Proposition \ref{propAssociatedBratteliDiagram}(b), there is a $ e_{ k_n +  1 } \in E_{ k_n + 1 , \min } $ and a $ v_{ k_n + 1 } \in V_{ k_n + 1 , \min } $ with $ s ( e ) = v $ and $ r ( e ) = v_{ k_n + 1 } $. Proceeding inductively, we construct the desired result, with $ v' = v_{ k_{ n + 1 } } $. The same argument works for $ V_{ \max }' $ in place of $ V_{ \min }' $. This proves that conclusion (b) holds.

Next, we prove that conclusion (c) holds. Let $ v \in V_{ n , \min }' $ and let $ e \in E_{ n + 1 , \min }' $ satisfy $ r ( e ) \in R ( v ) $. We want to show that $ s ( e ) = v $. We once again regard all of this as happening in $ B $ instead of $ B' $. What this means is that we have $ v \in V_{ k_n , \min }' $ and a minimal path $ ( e_{ k_n + 1 } , \ldots , e_{ k_{ n + 1 } } ) \in P_{ k_n + 1 , k_{ n + 1 } } $ such that there is some path $ ( e_{ k_n + 1 }' , \ldots , e_{ k_{ n + 1 } }' ) \in P_{ k_n + 1 , k_{ n + 1 } } $ such that $ s ( e_{ k_n + 1 }' ) = v $ and $ r ( e_{ k_{ n + 1 } }' ) = r ( e_{ k_{ n + 1 } } ) $, and we want to show that $ s ( e_{ k_n + 1 } ) = v $. Suppose not. Then by Proposition \ref{propAssociatedBratteliDiagram}(c), we have $ r ( e_{ k_n + 1 } ) \notin R ( v ) $; in particular, we have $ r ( e_{ k_n + 1 } ) \neq r ( e_{ k_n + 1 }' ) $. Proceeding like this, we eventually see $ r ( e_{ k_{ n + 1 } } ) \neq r ( e_{ k_{ n + 1 } }' ) $, a contradiction. Thus, $ s ( e_{ k_n + 1 } ) = v $. The proof for $ V_{ \max }' $ in place of $ V_{ \min }' $ is analogous. This proves that conclusion (c) holds.

That conclusion (d) holds is immediate; replace $ m $ with $ k_{ n + m } - k_n $. This completes the proof of the proposition.
\end{proof}

\section{The Dynamical Classification Theorem}

We now prove our main theorems, Theorem \ref{thmMainTheoremPrelude} and Theorem \ref{thmMainTheorem}. 

\begin{lemma}\label{lemmaPassesThoughSameVertex}
Let $ B = ( V , E , \leq ) $ be an ordered Bratteli diagram with a perfect ordering. Let $ e , f \in X_B $ and suppose $ e $ and $ f $ pass through the same vertex $ v $ at level $ k $. Then there is some $ N \in \bZ $ such that $ h_B^N ( e ) = f $.
\end{lemma}

\begin{proof}
Write $ e = ( e_1 , e_2 , \ldots ) $ and write $ f = ( f_1 , f_2 , \ldots ) $. Let $ n_1 $ be the largest element of $ \{ 1 , \ldots , k \} $ such that $ e_{ n_1 } \neq f_{ n_1 } $. 

Suppose $ e_{ n_1 } < f_{ n_1 } $. Then there is an $ N_1 \in \bZ_{ > 0 } $ such that 
\[
h_B^{ N_1 } ( e ) = ( e_1' , \ldots , e_{ n_1 - 1 }' , f_{ n_1 } , e_{ n_1 + 1 } , e_{ n_1 + 2 } , \ldots ) ,
\]
where $ ( e_1' , \ldots , e_{ n_1 - 1 }' ) $ is the minimal path from $ v_0 \in V_0 $ to $ s ( f_{ n_1 } ) $. Now, $ h_B^{ M } ( e ) $ and $ f $ pass through the same vertex at level $ n_1 - 1 $. So let $ n_2 $ be the largest element of $ \{ 1 , \ldots , n_1 - 1 \} $ such that $ e_{ n_2 }' \neq f_{ n_2 } $. Clearly $ e_{ n_2 }' < f_{ n_2 } $. So repeating the above process, we find an integer $ N_2 \in \bZ_{ > 0 } $ such that $ h_B^{ N_1 + N_2 } ( e ) = ( e_1\dprime , \ldots , e_{ n_2 - 1 }\dprime , f_{ n_2 } , \ldots , f_{ n_1 } , e_{ n_1 + 1 } , \ldots ) $. Repeating this process inductively, we arrive at an integer $ N $ such that $ h_B^N ( e ) = f $.
\end{proof}

\begin{thm}\label{thmMainTheoremPrelude}
Let $ ( X_1 , h_1 , Z_1 ) $ and $ ( X_2 , h_2 , Z_2 ) $ be fiberwise essentially minimal zero-dimensional systems. Then $ ( X_1 , h_1 , Z_1 ) $ and $ ( X_2 , h_2 , Z_2 ) $ are strong orbit equivalent if and only if 
\[
( K_0 ( C^* ( \bZ , X_1 , h_1 ) ) , K_0 ( C^* ( \bZ , X_1 , h_1 ) )^+ , 1 ) \cong ( K_0 ( C^* ( \bZ , X_2 , h_2 ) ) , K_0 ( C^* ( \bZ , X_2 , h_2 ) )^+ , 1 ) 
\]
and 
\[
K_1 ( C^* ( \bZ , X_1 , h_1 ) ) \cong K_1 ( C^* ( \bZ , X_2 , h_2 ) ) .
\]
\end{thm}

\begin{proof}

($ \Leftarrow $). Let $ B_1 $ and $ B_2 $ be the Bratteli diagrams satisfying the conclusions of Proposition \ref{propAssociatedBratteliDiagram} for $ ( X_1 , h_1 , Z_1 ) $ and $ ( X_2 , h_2 , Z_2 ) $, respectively. By Theorem \ref{thmK0isomK0}, we have $ K^0 ( X_1 , h_1 ) \cong K^0 ( X_2, h_2 ) $. By Proposition \ref{propAssociatedBratteliDiagram}(a), we have $ K^0 ( X_{ B_1 } , h_{ B_1 } ) \cong K^0 ( X_{ B_2 } , h_{ B_2 } ) $. By a slight but trivial extension of Theorem 5.4 in \cite{HermanPutnamSkau92}, we have $ K_0 ( B_1 ) \cong K_0 ( B_2 ) $. By \cite{Elliott76}, we have $ B_1 \sim B_2 $ (in the equivalence class of Bratteli diagrams generated by telescoping and isomorphism), which tells us that there is a (non-ordered) Bratteli diagram $ B $ such that telescoping $ B $ to odd levels yields a telescoping of $ B_1 $ and telescoping $ B $ to even levels yields a telescoping of $ B_2 $. By replacing $ B_1 $ and $ B_2 $ with their telescopings (which can be done without changing the above due to Proposition \ref{propTelescopingPreservesAssociatedDiagramProperties}), we may assume that telescoping $ B $ to odd levels yields $ B_1 $ and telescoping $ B_2 $ to even levels yields $ B_2 $. Let $ B' $ be the telescoping of $ B $ by the sequence $ ( 3n - 2 ) $, so that telescoping $ B $ to odd levels yields a telescoping of $ B_1 $ by $ ( 3n - 2 ) $ and telescoping $ B $ to even levels yields a telescoping of $ B_2 $ by $ ( 3n - 1 ) $. Note that by Proposition \ref{propTelescopingPreservesAssociatedDiagramProperties}, these telescopings of $ B_1 $ and $ B_2 $ also satisfy the conclusions of Proposition \ref{propAssociatedBratteliDiagram}.

We denote by $ V_{ \min }' $ ($ V_{ \max }' $) the minimal (resp. maximal) vertices as inherited by $ B_1 $ and $ B_2 $. We claim that $ B' $ has the following properties:
\begin{enumerate}[(i)]
\item Let $ v \in V_{ \min }' $. There is at least one $ v' \in V_{ \min }' $ with $ v' \in R ( v ) $.
\item Let $ v \in V_{ \min }' $. There is precisely one $ v' \in V_{ \min }' $ with $ v' \in S ( v ) $.
\end{enumerate}

We first prove property (i). Let $ v \in V_{ \min , n }' $. The statement is trivially true for $ n = 0 $ so suppose $ n > 0 $ and without loss of generality suppose that $ n $ is odd. View $ v $ as a vertex in $ V_{ \min , 3n - 2 } $, so that the statement we are trying to prove becomes $ R^3 ( v ) \cap V_{ \min } \neq \varnothing $. Notice that for every vertex $ w \in V^{ ( 2 ) } \setminus V^{ ( 2 ) }_0 $, $ S ( w ) \cap V^{ ( 2 ) }_{ \min } \neq \varnothing $, since there is always an $ e \in E^{ ( 2 ) }_{ \min } $ with $ s ( e ) = w $. Thus, it follows that $ ( S^2 \circ R^3 ) ( v ) \cap V_{ \min } \neq \varnothing $. By Proposition \ref{propAssociatedBratteliDiagram}(b), we have $ ( R^2 \circ S^2 \circ R^3 ) ( v ) \cap V_{ \min } \neq \varnothing $. By Proposition \ref{propAssociatedBratteliDiagram}(d), we have $ R^2 \circ S^2 \circ R^2 = R^2 $. Thus, $ R^2 \circ S^2 \circ R^3 = R^3 $, so $ R^3 ( v ) \cap V_{ \min } \neq \varnothing $. The proof for $ v \in V_{ \max }' $ is analogous. Altogether, this proves property (i).

We now prove property (ii). Let $ v \in V_{ \min , n }' $ and without loss of generality suppose that $ n $ is odd. The case $ n = 1 $ is trivial so suppose $ n > 1 $. View $ v $ as a vertex in $ V_{ \min , 3n - 2 } $, so that the statement we are trying to prove is that $ S^3 ( v ) $ contains precisely one element. Let $ w \in S ( v ) $ and let $ ( e_1 , e_2 ) $ be a minimal path with $ r ( e_2 ) = w $. Since $ s ( e_1 ) \in V_{ \min } $, we have $ S^3 ( v ) \cap V_{ \min } \neq \varnothing $. Now, suppose $ S^3 ( v ) \cap V_{ \min } $ contains two elements, $ v_1 $ and $ v_2 $. By Proposition \ref{propAssociatedBratteliDiagram}(c), we have $ R^2 ( v_1 ) \cap R^2 ( v_2 ) = \varnothing $. By Proposition \ref{propAssociatedBratteliDiagram}(b), $ R^2 ( v_1 ) \cap V_{ \min } \neq \varnothing $ and $ R^2 ( v_2 ) \cap V_{ \min } \neq \varnothing $. By Proposition \ref{propAssociatedBratteliDiagram}(c) again, 
\begin{equation}\label{eq:lastproof1}
R^2 ( R^2 ( v_1 ) \cap V_{ \min } ) \cap R^2 ( R^2 ( v_2 ) \cap V_{ \min } ) = \varnothing 
\end{equation}
Since $ v \in R^3 ( v_1 ) $ and $ v \in R^3 ( v_2 ) $, we have 
\begin{equation}\label{eq:lastproof2}
R^4 ( v_1 ) \cap R^4 ( v_2 ) \neq \varnothing .
\end{equation}
But now notice that by Proposition \ref{propAssociatedBratteliDiagram}(d), we have $ R^4 ( v_1 ) = R^2 ( R^2 ( v_1 ) \cap V_{ \min } ) $ and $ R^4 ( v_2 ) = R^2 ( R^2 ( v_2 ) \cap V_{ \min } ) $. Thus, we see that (\ref{eq:lastproof1}) and (\ref{eq:lastproof2}) yield a contradiction. Therefore, $ S^3 ( v ) \cap V_{ \min } $ contains precisely one element. The proof for $ v \in V_{ \max }' $ is analogous. This proves property (ii).

For convenience, replace $ B^{ ( 1 ) } $ with its telescoping by $ ( 3 n - 2 ) $  and replace $ B^{ ( 2 ) } $ with its telescoping by $ ( 3 n - 1 ) $. Now, let $ e = ( e_1 , e_2 , \ldots ) \in X_{ B^{ ( 1 ) } , \min } $ and let $ ( v_1 , v_2 , \ldots ) $ be its associated vertices. Let $ n \in \bZ_{ > 1 } $. We view $ v_n $ as a vertex in $ V_{ 2n - 1 }' $. By property (ii), there is a unique vertex $ v_{ n - 1 }' \in V_{ \min , 2n - 2 }' $ such that $ v_{ n - 1 }' \in S ( v_n ) $. We claim that $ v_{ n - 1 } \in S ( v_{ n - 1 }' ) $. If not, then by property (ii) there is some $ w \in S ( v_{ n - 1 }' ) $ such that $ w \neq v_{ n - 1 } $. But then $ R^2 ( w ) \cap R^2 ( v_{ n - 1 } ) \neq \varnothing $, which is a contradiction again by Proposition \ref{propAssociatedBratteliDiagram}(c). Thus, for each $ n \in \bZ_{ > 1 } $, $ v_{ n - 1 }' $ is connected by a path to $ v_n' $, there is an $ e_n' \in E_{ \min  , n }^{ ( 2 ) } $ with $ s ( e_n' ) = v_{ n - 1 }' $ and $ r ( e_n' ) = v_n' $ by Proposition \ref{propAssociatedBratteliDiagram}(c). Thus, this gives us $ e' = ( e_1' , e_2' , \ldots ) \in X_{ B^{ ( 2 ) } , \min } $ such that for each $ n \in \bZ_{ > 0 } $, we have $ s ( e_n' ) = v_{ n - 1 }' $ and $ r ( e_n' ) = v_n' $.

Let $ f = ( f_1 , f_2 , \ldots ) \in X_{ B' } $ be any path with $ r ( f_n ) = s ( f_{ n + 1 } ) = v_{ 2n - 1 } $ for all odd $ n \in \bZ_{ > 0 } $ and $ r ( f_n ) = s ( f_{ n + 1 } ) = v_{ 2n }' $ for all even $ n \in \bZ_{ > 0 } $. Since this respects the vertices and the range and source maps, we are free to define $ F_1 (  e ) = f $ and $ F_2 ( e' ) = f $.

We now show that the above pairing is a bijection between $ X_{ B , \min }^{ ( 1 ) } $ and $ X_{ B , \min }^{ ( 2 ) } $. Let $ e' = ( e_1' , e_2' , \ldots ) \in X_{ B^{ ( 2 ) } , \min } $ and $ ( v_n' ) $ be as above. By property (ii), for each $ n \in \bZ_{ > 0 } $ there is precisely one $ w_n \in V_{ \min }' $ with $ w_n \in S ( v_n' ) $. Since $ v_n \in S ( v_n' ) $, we must have $ w_n = v_n $. Thus, the bijection is established, and so $ F_1 ( X_{ B^{ ( 1 ) } , \min } ) = F_2 ( X_{ B^{ ( 2 ) } , \min } ) $. We now repeat the process for maximal vertices, choosing edges which are not minimal if the number of edges between the vertices is more than one. This extends $ F_1 $ and $ F_2 $ so that $ F_1 ( X_{ B^{ ( 1 ) } , \max } ) = F_2 ( X_{ B^{ ( 2 ) } , \max } ) $.

We now extend $ F_1 $ and $ F_2 $ by any bijection between $ E_n^{ ( 1 ) } $ and $ E_{ 2 n - 1 }' $ and any bijection between $ E_n^{ ( 2 ) } $ and $ E_{ 2 n } $ that respect the range and source maps. In this way, we get homeomorphisms $ F_1 : X_{ B^{ ( 1 ) } } \to X_{ B' } $ and $ F_2 : X_{ B^{ ( 2 ) } } \to X_{ B' } $. Define $ F = F_2^{ -1 } \circ F_1 : X_{ B^{ ( 1 ) } } \to X_{ B^{ ( 2 ) } } $.

Let $ \alpha , \beta : X_{ B^{ ( 1 ) } } \to \bZ $ be the orbit cocyles of $ F $. We will show that $ \alpha $ and $ \beta $ are continuous on $ X_{ B^{ ( 1 ) } } \setminus X_{ B^{ ( 1 ) } , \max } $. So let $ e = ( e_1 , e_2 , \ldots ) \in X_{ B^{ ( 1 ) } } \setminus X_{ B^{ ( 1 ) } , \max } $. Let $ k $ be the smallest element of $ \bZ_{ > 0 } $ such that $ e_k \notin E^{ ( 1 ) }_{ \max } $.  Then $ e $ and $ h_{ B^{ ( 1 ) } } ( e ) $ are confinal from level $ k $. This means that $ F_1 ( e ) $ and $ F_1 ( h_{ B^{ ( 1 ) } } ( e )  ) $ are cofinal from level $ 2 k - 1 $, and so $ F ( e ) $ and $ F ( h_{ B^{ ( 1 ) } } ( e ) ) $ are cofinal from level $ k $. In particular, $ F ( e ) $ and $ F ( h_{ B^{ ( 1 ) } } ( e ) ) $ pass through the same vertex $ v $ at level $ k $. By Lemma \ref{lemmaPassesThoughSameVertex}, there is an integer $ N $ such that 
\[
h_{ B^{ ( 1 ) } }^N ( e ) = F ( h_{ B^{ ( 1 ) } } ( e ) ) 
\]
Let $ f \in U ( e_1 , \ldots , e_{ k + 1 } ) $. Then $ y $ and $ h_{ B^{ ( 1 ) } } ( f ) $ are confinal from level $ k $. This means that $ F_1 ( f ) $ and $ F_1 ( h_{ B^{ ( 1 ) } } ( f )  ) $ are cofinal from level $ 2 k - 1 $, and so $ F ( y ) $ and $ F ( h_{ B^{ ( 1 ) } } ( f )  ) $ are cofinal from level $ k $. Since $ e , f \in  U ( e_1 , \ldots , e_{ k + 1 } ) $, $ h_{ B^{ ( 1 ) } } ( x ) $ and $ h_{ B^{ ( 1 ) } } ( f ) $ have the same initial segment from level $ 0 $ to level $ k + 1 $, and so $ F ( h_{ B^{ ( 1 ) } } ( e ) ) $ and $ F ( h_{ B^{ ( 1 ) } } ( f ) ) $ have the same initial segment from level $ 0 $ to level $ k $. Similarly, $ F ( e ) $ and $ F ( f ) $ have the same initial segment from level $ 0 $ to level $ k $. Thus, the integer $ N $ from above satisfies 
\[
h_{ B^{ ( 1 ) } }^N ( f ) = F ( h_{ B^{ ( 1 ) } } ( f ) ) 
\]
Since $ f \in U ( e_1 , \ldots , e_{ k + 1 } ) $ was arbitrary, this shows that $ \alpha $ is continuous at $ e $.

The argument for $ \beta $ is analogous to $ \alpha $. Thus, $ ( X_{ B^{ ( 1 ) } } , h_{ B^{ ( 1 ) } } , X_{ B^{ ( 1 ) } , \max } ) $ and $ ( X_{ B^{ ( 2 ) } } , h_{ B^{ ( 2 ) } } , X_{ B^{ ( 2 ) } , \max } ) $ are strong orbit equivalent. By replacing $ F $ with $ h_{ B^{ ( 1 ) } }^{ - 1 } \circ F $, we see that $ ( X_{ B^{ ( 1 ) } } , h_{ B^{ ( 1 ) } } , X_{ B^{ ( 1 ) } , \min } ) $ and $ ( X_{ B^{ ( 2 ) } } , h_{ B^{ ( 2 ) } } , X_{ B^{ ( 2 ) } , \min } ) $ are strong orbit equivalent. Therefore, by Proposition \ref{propAssociatedBratteliDiagram}(a), $ ( X_1 , h_1 . Z_1 ) $ and $ ( X_2 , h_2 , Z_2 ) $ are strong orbit equivalent.

($\Rightarrow$). Assume $ ( X_1 , h_1 , Z_1 ) $ and $ ( X_2 , h_2 , Z_2 ) $ are strong orbit equivalent. By the definition of strong orbit equivalence, we have $ Z_1 \cong Z_2 $ and so by Theorem \ref{thmK1Formula}, we have $ K_1 ( C^* ( \bZ , X_1 , h_1 ) ) \cong K_1 ( C^* ( X_2 , h_2 , Z_2 ) ) $.

Let $ \beta , \gamma : X_1 \to \bZ $ be the associated orbit cocyles defined by $ F \circ h_1 = h_2^\beta \circ F $ and $ F \circ h_1^\gamma = h_2 \circ F $. Recall that strong orbit equivalence tells us that $ \beta $ and $ \gamma $ are continuous on $ X_1 \setminus Z_1 $ and $ F ( Z_1 ) = Z_2 $. 

Let $ \widetilde{ h }_2 = F^{ - 1 } \circ h_2 \circ F : X_1 \to X_1 $. Then $ \widetilde{ h }_2 $ is conjugate to $ h_2 $, has the same orbits as $ h_1 $, and
\[
h_1 = \widetilde{ h }_2^\beta , \widetilde{ h }_2 = h_1^\gamma.
\]
We also see that $ ( X_1 , \widetilde{ h }_2 , Z_1 ) $ is a fiberwise essentially minimal zero-dimensional system that is conjugate to $ ( X_2 , h_2 , Z_2 ) $, so we work with the former for the remainder of the proof.

Let $ \alpha_1 $ be the automorphism of $ C ( X_1 ) $ induced by $ h_1 $ and let $ \alpha_2 $ be the automorphism of $ C ( X_1 ) $ induced by $ \widetilde{ h }_2 $. We now show that $ \ran ( \id - ( \alpha_1 )_* ) \subset \ran ( \id - ( \alpha_2 )_* ) $. It is enough to show that for any compact open set $ E \subset X_1 $, we have $ ( \id - ( \alpha_1 )_* ) ( \chi_E ) \in \ran ( \id - ( \alpha_2 )_* ) $. 

Given $ f \in C ( X_1 , \bZ ) $, we denote the image of $ f $ in $ K_0 ( C^* ( \bZ , X_1 , h_i ) ) $ by $ [ f ]_1 $ and denote the image of $ f $ in $ K_0 ( C^* ( \bZ , X_1 , \widetilde{ h }_2 ) ) $ by $ [ f ]_2 $. Since $ h_1 $ and $ \widetilde{ h }_2 $ have the same orbits, there is a probably measure $ \mu $ that is both $ h_1 $- and $ \widetilde{ h }_2 $-invariant. Note that $ C^* ( \bZ , X_1 , h_1 ) $ is the $ C^* $-subalgebra of $ L ( L^2 ( X_1 , \mu ) ) $ generated by $ C ( X ) $ and the unitary $ u_1 : g \to g \circ h_1^{ -1 } $. Also note that $ C^* ( \bZ , X_1 , \widetilde{ h }_2 ) $ is the $ C^* $-subalgebra of $ L ( L^2 ( X_1 , \mu ) ) $ generated by $ C ( X ) $ and the unitary $ u_2 : g \to g \circ \widetilde{ h }_2^{ -1 } $. For the remainder of the proof, we identify these crossed products with these corresponding subalgebras of $ L ( L^2 ( X_1 , \mu ) ) $.

We claim that map $ \vphi : K_0 ( C^* ( \bZ , X_1 , h_1 ) ) \to K_0 ( C^* ( \bZ , X_1 , \widetilde{ h }_2 ) ) $ defined by $ \vphi ( [ \chi_U ]_1 ) = [ \chi_U ]_2 $ is an isomorphism of ordered groups. Since the positive cone and distinguished order units agree via this map, the only thing to check is that $ \image ( \id - ( \alpha_1 )_* ) = \image ( \id - ( \alpha_2 )_* ) $.

Let $ E $ be a compact open subset of $ X_1 $ such that $ E \cap Z_1 = \varnothing $. Then since $ \beta $ is continuous on $ E $, $ \ran ( \beta|_E ) = \{ k_1 , \ldots , k_N \} $. Then
\[
u_1 \chi_E = \sum_{ n = 1 }^N u_2^{ k_n } \chi_{ E \cap \beta^{ -1 } ( k_n ) } .
\]
Thus, $ u_1 \chi_E \in C^* ( \bZ , X_1 , \widetilde{ h }_2 ) $, and, letting $ A_{ Z_1 }^{ ( 1 ) } $ correspond to $ C^* ( \bZ , X_1 , h_1 ) $ as in Definition \ref{defnAZ}, we have $ A_{ Z_1 }^{ ( 1 ) } \subset C^* ( \bZ , X_1 , \widetilde{ h }_2 ) $. Since $ E \cap Z_1 = \varnothing $, there is a unitary $ v \in A_{ Z_1 }^{ ( 1 ) } $ such that $ v \chi_E v^* = \chi_{ h_1 ( E ) } $. Since we also have $ v \in C^* ( \bZ , X_1 , \widetilde{ h }_2 ) $, we have $ [ \chi_E ]_2 = [ \chi_{ h_1 ( E ) } ]_2 $, so $ \chi_E - \chi_{ h_1 ( E ) } \in \image ( \id - ( \alpha_2 )_* ) $. Theorem \ref{thmK0AZIsomorphism} (and the proof) tells us that $ K_0 ( A_{ Z_1 } )^{ ( 1 ) } \cong K_0 ( \bZ , X_1 , h_1 ) $ by the injection map. Thus, $ \image ( \id - ( \alpha_1 )_* ) \subset \image ( \id - ( \alpha_2 )_* ) $. By repeating the above process with $ \gamma $ instead of $ \beta $, we can similarly show that $ \image ( \id - ( \alpha_2 )_* ) \subset \image ( \id - ( \alpha_1 )_* ) $. This completes the proof.
\end{proof}

\begin{thm}\label{thmMainTheorem}
Let $ ( X_1 , h_1 , Z_1 ) $ and $ ( X_2 , h_2 , Z_2 ) $ be fiberwise essentially minimal zero-dimensional systems with no periodic points. The following are equivalent:
\begin{enumerate}[(a)]
\item $ C^* ( \bZ , X_1 , h_1 ) \cong C^* ( \bZ , X_2 , h_2 ) $.
\item $ ( K_0 ( C^* ( \bZ , X_1 , h_1 ) ) , K_0 ( C^* ( \bZ , X_1 , h_1 ) )^+ , 1 ) \cong ( K_0 ( C^* ( \bZ , X_2 , h_2 ) ) , K_0 ( C^* ( \bZ , X_2 , h_2 ) )^+ , 1 ) $ and  

\noindent $ K_1 ( C^* ( \bZ , X_1 , h_1 ) ) \cong K_1 ( C^* ( \bZ , X_2 , h_2 ) ) $.
\item $ ( X_1 , h_1 , Z_1 ) $ and $ ( X_2 , h_2 , Z_2 ) $ are strong orbit equivalent.
\end{enumerate}
\end{thm}

\begin{proof}

(a) $ \Longleftrightarrow $ (b). By Theorem 2.2 and 2.3 of \cite{Herstedt21}, $ C^* ( \bZ , X_1 , h_1 ) $ and $ C^* ( \bZ , X_2 , h_2 ) $ are A$\bT$-algebras of real rank zero, so this result follows from \cite{DadarlatGong97}.

(b) $ \Longleftrightarrow $ (c). This is implied by Theorem \ref{thmMainTheoremPrelude}.

\end{proof}

\bibliographystyle{plain}
\bibliography{researchbib}

\end{document}